\documentclass[oneside,english,letterpaper]{amsart}
\usepackage[T1]{fontenc}
\usepackage[utf8]{inputenc}
\usepackage{lmodern}
\usepackage[english]{babel}
\usepackage{amsmath,amsthm,amssymb,mathrsfs}
\usepackage{enumitem,xcolor,microtype,csquotes}
\providecommand{\IfDocumentMetadataT}[1]{}

\usepackage[backend=biber,style=numeric,sorting=nyt,
  doi=false,url=false,maxbibnames=10]{biblatex}
\AtEveryBibitem{\clearfield{note}}
\usepackage[hidelinks,pdfusetitle,bookmarksopen=true]{hyperref}
\hypersetup{pdftitle={Determinantal capacity and L-infinity estimates}}
\usepackage{xurl}
\numberwithin{equation}{section}

\theoremstyle{plain}
\newtheorem{thm}{Theorem}[section]
\newtheorem{lem}[thm]{Lemma}
\newtheorem{prop}[thm]{Proposition}
\newtheorem{cor}[thm]{Corollary}
\theoremstyle{definition}
\newtheorem{defn}[thm]{Definition}
\newtheorem{example}[thm]{Example}

\theoremstyle{remark}
\newtheorem{rem}[thm]{Remark}

\def\pdv{\partial}
\newcommand{\gar}{\ifmmode\mathrm{G\mathring{a}rding}\else G{\aa}rding\fi}
\renewcommand{\d}{\mathrm{d}}
\newcommand{\ddc}{\d\d^c}
\def\R{\mathbb R}
\def\N{\mathbb N}
\def\x{\mathbf{x}}

\def\v{\mathbf{v}}
\def\one{\mathbf{1}}
\def\D{\mathrm D}
\def\Newt{\operatorname{Newt}}
\def\supp{\operatorname{supp}}
\def\Capacity{\operatorname{Cap}}
\def\tr{\operatorname{tr}}
\def\GS{\mathrm G}
\def\pol{\Pi^\uparrow}

\title {Determinantal capacity and $L^\infty$ estimates}

\begin{document}

\author{Hao Fang}
\address{Department of Mathematics, University of Iowa, Iowa City, IA 52246,
USA}
\email{hao-fang@uiowa.edu}
\author{Biao Ma}
\address{School of Mathematical Sciences, East China Normal University, 500
Dongchuan Road, Shanghai, China.}
\email{bma@math.ecnu.edu.cn}
\thanks{H. F.'s work is partially supported by a Simons Foundation mathematics
collaboration grant. B. M. is partially supported
by NSFC grant (No. 12471052). J. Y. Wu is partially supported by NSFC grant (No. 12371207).
}

\author{Jinyang Wu}
\address{Institute of Mathematical Sciences (IMS), ShanghaiTech University, 393 Middle Huaxia Road, Shanghai, China}
\email{wujy3@shanghaitech.edu.cn}

\begin{abstract}
We introduce determinantal ellipticity, or det-ellipticity, a quantitative structural
condition for fully nonlinear elliptic operators on compact K\"ahler
manifolds that provides the link between
the $L^\infty$-estimates and algebraic/combinatorial properties of a large class of Hessian elliptic operators.

On the analytic side, we introduce $\mathcal D$-subsolutions extending the determinant sublevel condition of Sui--Sun.  Using the auxiliary comparison method of Guo--Phong--Tong
and pluripotential theory for complex Monge--Amp\`ere equations, we obtain relative $L^\infty$-estimates in big cohomology classes under a determinant-entropy bound for viscosity supersolutions and singular reference potentials. Det-ellipticity
provides a systematic construction of such subsolutions.

On the algebraic side, we study G{\aa}rding elliptic polynomial
operators of degree $d$. We characterize the determinant increment
bound with exponent $d/n$ by positivity of the determinantal capacity
of the top homogeneous part. This is equivalent to balanced-point
conditions for the Newton polytopes of rank-one scalarizations and
to slope semistability of an associated subspace polymatroid.
Together with the analytic hypotheses, these criteria yield relative
$L^\infty$ estimates for the corresponding fully nonlinear equations.
\end{abstract}
\maketitle

\section{Introduction}
\label{sec:introduction}
In this paper, we study a priori $L^\infty$ estimates, which play a fundamental role in the analysis of fully nonlinear complex Hessian equations.

Let $(X,\rho)$ be a compact K\"ahler
manifold of complex dimension $n$, let $\mathcal A$ be an admissible
subset of the bundle of real $(1,1)$-forms, and let $F$
be a smooth elliptic operator on $\mathcal A$. We consider  admissible
solutions of
\begin{equation}
\label{eq:intro-equation}
F_x(\omega+dd^c u)=\psi(x),
\qquad
\omega+dd^c u\in\mathcal A_x,
\qquad
\sup_Xu=0.
\end{equation}

Many classical geometric PDEs fall into this category. We concentrate on the
$L^\infty$ estimates. For the complex Monge--Amp\`ere   equation, Yau obtained
the required $L^\infty$ estimate by Moser iteration \cite{MR480350}, while
Ko{\l}odziej later established the sharp estimate for right-hand sides in
$L^p$, with $p>1$, using pluripotential theory \cite{MR1618325}.
B{\l}ocki developed a different approach based on the
Alexandrov--Bakelman--Pucci (ABP) maximum principle \cite{MR2156505}.
More recently, Guo, Phong, and Tong introduced an auxiliary
Monge--Amp\`ere method combining a pointwise comparison, exponential
integrability, entropy estimates, and a De Giorgi iteration
\cite{MR4593734}. This method was subsequently adapted by
Guo and Phong to general fully nonlinear equations
\cite{guo2024uniformlinftyestimatessubsolutions}.

Sui--Sun~\cite[condition~(3)]{MR4567566} formulated a determinant
sublevel condition for spectral equations, bounding the determinant
of positive increments in a prescribed sublevel set without
bounding their individual eigenvalues. This weakens the classical
$\mathcal C$-subsolution condition of Guan and
Sz\'ekelyhidi~\cite{MR3284698,MR4727578,MR3807322}.
Under this condition and their structural hypotheses, they proved
$L^\infty$ estimates for general spectral equations when the
auxiliary class is K\"ahler~\cite[Theorem~1.1]{MR4567566}.
Using smooth auxiliary Monge--Amp\`ere equations on K\"ahler
approximations, they further obtained relative $L^\infty$ estimates
for Hessian quotient equations with nef and big auxiliary
classes~\cite[Theorem~1.2]{MR4567566}.

In the analytic part of our work, we formulate determinant
subsolutions for general admissible pairs, without a spectral
assumption. Our main observation is that the auxiliary
Monge--Amp\`ere comparison method
\cite{MR4593734,guo2024uniformlinftyestimatessubsolutions}
can be combined directly with pluripotential theory:
we solve the auxiliary equations in the non-pluripolar sense
and apply local pluripotential--viscosity comparison.
This yields relative estimates in arbitrary big classes and
accommodates viscosity supersolutions and singular reference
potentials.

Let $\chi$ be a smooth closed real $(1,1)$-form and let
$\gamma\in C(X,\mathbb R)$. We call $(\omega,\chi,\gamma)$ a
\emph{determinantal subsolution}, or a $\mathcal D$-subsolution,
of \eqref{eq:intro-equation} if, for every $x\in X$ and
$P\in\mathcal Q_x^+$, the conditions
\begin{equation}
\label{eq:intro-D-condition}
\omega(x)-\chi(x)+P\in\mathcal A_x,
\qquad
F_x\bigl(\omega(x)-\chi(x)+P\bigr)\leq\psi(x)
\end{equation}
imply
\begin{equation}
\label{eq:intro-D-bound}
\det_{\rho(x)}P\leq e^{\gamma(x)}.
\end{equation}

Let $\chi$ be a smooth closed real $(1,1)$-form on $X$.
We impose only the cohomological positivity condition that
$[\chi]$ is big; neither nefness of $[\chi]$ nor pointwise
positivity of $\chi$ is required. Write
\begin{equation}
\label{eq:intro-volume}
V_\chi:=\operatorname{vol}([\chi])>0
\end{equation}
and let
\begin{equation}
\label{eq:intro-envelope}
W_\chi
:=
\left(
\sup\{w\in\operatorname{PSH}(X,\chi):w\leq0\}
\right)^*
\end{equation}
be its canonical envelope with minimal singularities.
For $q>0$, define the determinant entropy by
\begin{equation}
\label{eq:intro-D-entropy}
\operatorname{Ent}_{\mathcal D,q}(\chi,\gamma,\rho)
:=
\int_X
\left(1+(\gamma-\log V_\chi)_+\right)^q
\frac{e^\gamma}{V_\chi}\rho^n.
\end{equation}

\begin{thm}
\label{thm:intro-relative-Linfty}
Fix a K\"ahler form $\rho$ on $X$ and let $q>n$.
Let $(F,\mathcal A)$ be an admissible pair satisfying QSH,
and let $u\in C^\infty(X)$ be an admissible solution of
\eqref{eq:intro-equation}. Suppose that $(\omega,\chi,\gamma)$
is a $\mathcal D$-subsolution, where $\chi$ is a smooth closed
real $(1,1)$-form with $[\chi]$ big and
$\gamma\in C(X,\mathbb R)$.
If
\[
\operatorname{Ent}_{\mathcal D,q}(\chi,\gamma,\rho)\leq H,
\]
then
\[
\sup_X(W_\chi-u)\leq C,
\]
where
\[
C=C\left(
X,\rho,\chi,q,H,
b_{\mathcal A}
+\|\operatorname{tr}_\rho\omega\|_{L^\infty(X)}
\right).
\]
If $\chi\geq0$, then $W_\chi=0$ and
$\|u\|_{L^\infty(X)}\leq C$.
\end{thm}

Here QSH denotes uniform quasi-subharmonicity
(Definition~\ref{def:QSH}); it supplies the normalized $L^1$
bound through Lemma~\ref{lem:admissible-L1}.
The ABP estimate in Appendix~\ref{ABP} also applies under the
$\mathcal D$-subsolution condition, but requires
$\chi\geq\delta\rho$ and depends on this positivity constant.
The relative estimate above requires no such lower bound.
Its viscosity and singular formulations are given in
Section~\ref{sec:relative-Linfty}.

Theorem \ref{thm:intro-relative-Linfty} recovers some of the smooth spectral results of
Sui--Sun~\cite{MR4567566}, including their relative estimate for
Hessian quotient equations in nef and big auxiliary classes.
Our results further apply to nonspectral admissible pairs, big non-nef
auxiliary classes, viscosity supersolutions, and singular
reference potentials.

Our second observation is the structural condition of
\emph{det-ellipticity}, which bridges the analytic and algebraic parts of
the paper.  An admissible pair $(F,\mathcal A)$ is det-elliptic if there
exist $\eta,\beta>0$ such that
$$
F_x(A+P)-F_x(A)\geq\eta\,\det_\rho(P)^\beta
$$
for $A\in\mathcal A_x$ and $P\in\mathcal Q_x^{++}$.  This pointwise
coercivity provides the determinant control underlying the
$\mathcal D$-subsolutions and hence the relative $L^\infty$ estimate.
Its precise definition and analytic consequences are given in Section~\ref{sec:admissible-equations}.

Motivated by the search for admissible equations to which
Theorem~\ref{thm:intro-relative-Linfty} applies, we introduce a broad class of operators that
contains many of the standard examples in complex geometric PDE.

The algebraic framework  of the paper is \emph{$\mathcal K$-G{\aa}rding polynomials},    a conic
extension of the theory of G{\aa}rding polynomials developed in \cite{Fang-MaGard} by the first and second-named authors.   The \gar{} polynomial theory studies
polynomials on a real vector space $W$  and is organized around a
distinguished positivity component that satisfies a positive ray test.
Motivated by elliptic PDE, we replace the positive orthant by an arbitrary
closed convex cone $\mathcal K\subset W$ and introduce
$\mathcal K$-G{\aa}rding polynomials, with the corresponding component and
derivative-nesting conditions imposed in directions of $\mathcal K$.
The original  \gar{} theory is recovered as a special case.  An affine scalarization
theorem and a top-part closure result provide the bridge back to the earlier G{\aa}rding-polynomial
theory, making its structural results available for the study of the conic
setting.

This conic formulation is naturally adapted to complex elliptic PDE.
Let $V$ be a complex vector space, let $\mathcal Q(V)$ denote the real
vector space of Hermitian forms on $V$, and take
$\mathcal K=\mathcal Q(V)^+$, the cone of semipositive forms.  We call a
polynomial
\[
G:\mathcal Q(V)\longrightarrow\mathbb R
\]
a \emph{G{\aa}rding elliptic operator} if it is
$\mathcal Q(V)^+$-G{\aa}rding.  
Such an operator has a distinguished positivity component
$\mathcal C_G$ and a first-derived set $\mathcal C_G^{(1)}$.
We use the open region
\[
\mathcal A_G:=\operatorname{int}\mathcal C_G^{(1)},
\qquad \mathcal C_G\subseteq\mathcal A_G,
\]
on which $(G,\mathcal A_G)$ satisfies the admissible-pair conditions;
see Proposition~\ref{prop:positive-branch-admissible}.
The analytic application also requires QSH on the chosen region.

The class of G{\aa}rding elliptic operators includes the complex
Monge--Amp\`ere, complex Hessian, and mixed-form operators. For polarized spectral polynomials, we also
prove admissibility, QSH, and a determinant increment bound directly
in Proposition~\ref{thm:spectral-lift-polarization}.  

Our third observation is that, for an operator of degree $d$, the
determinant increment bound with exponent $d/n$ is characterized algebraically  by
capacity and polymatroid semistability.

Let $G$ be a G{\aa}rding elliptic polynomial operator of degree $d\geq1$,
let $
H:=G^{\mathrm{top}}$
be its top homogeneous part, and set {$\mathcal A=\operatorname{int}\mathcal C_G^{(1)}$}.
Define the determinantal capacity of $H$ by
\begin{equation}
\label{eq:intro-det-capacity}
\operatorname{Cap}_{\det}(H)
:=
\inf_{P\in\mathcal Q(V)^{++}}
\frac{H(P)}{\det_\rho(P)^{d/n}}.
\end{equation}

For a $\rho$-unitary basis
$\mathbf v=(v_1,\ldots,v_n)$, define the rank-one scalarization
\[
g_{\mathbf v}(x)
:=
H\left(
\sum_{i=1}^n x_i\,v_i\otimes\overline{v_i}
\right).
\]
For a complex subspace $U\subseteq V$, let $\pi_U$ be the
$\rho$-orthogonal projection onto $U$ and set
\[
r_H(U):=\deg_t H(I+t\pi_U).
\]
We prove that $r_H$ is normalized, nondecreasing, and submodular.
It therefore defines an integral subspace polymatroid
\[
\mathcal P_H:=(V,r_H).
\]

The core algebraic result of the paper is summarized in the following theorem.
\begin{thm}[Determinantal capacity and slope semistability]
\label{thm:intro-capacity-characterization}
Let $G$ be a G{\aa}rding elliptic polynomial operator of degree $d\geq1$,
let $H=G^{\mathrm{top}}$, and set $\mathcal A=\operatorname{int}\mathcal C_G^{(1)}$.
Then the following are equivalent:
\begin{enumerate}[label=\textup{(\arabic*)},nosep]
\item\label{enu:intro-capacity-1}
$(G,\mathcal A)$ is det-elliptic with exponent $d/n$: there exists
$\eta>0$ such that
\[
G(A+P)-G(A)\geq\eta\det_\rho(P)^{d/n}
\]
for every $A\in\mathcal A$ and $P\in\mathcal Q(V)^+$.
\item\label{enu:intro-capacity-2} For some $\eta>0$,
\[
H(P)\geq\eta\det_\rho(P)^{d/n}
\]
for every $P\in\mathcal Q(V)^+$.
\item\label{enu:intro-capacity-3} $\operatorname{Cap}_{\det}(H)>0$.
\item\label{enu:intro-capacity-4} For every $\rho$-unitary basis $\mathbf v$,
$\frac dn\mathbf 1\in\operatorname{Newt}(g_{\mathbf v})$.
\item\label{enu:intro-capacity-5} The subspace polymatroid $\mathcal P_H$ is slope semistable;
equivalently,
$r_H(U)\geq\frac dn\dim_{\mathbb C}U$
for every complex subspace $U\subseteq V$.
\end{enumerate}
Moreover, the optimal constant in \ref{enu:intro-capacity-1} and \ref{enu:intro-capacity-2} is
$\eta_{\mathrm{opt}}=\operatorname{Cap}_{\det}(H)$.
\end{thm}

Theorem~\ref{thm:intro-capacity-characterization} characterizes
det-ellipticity with exponent $d/n$ entirely through the top homogeneous
part. Its polymatroid formulation is intrinsic and identifies the
subspace degeneracies that obstruct this bound. 

Our determinantal capacity is related to Gurvits' capacity theory
for stable and hyperbolic homogeneous polynomials \cite{MR2411443}.
Harvey--Lawson proved determinant majorization for classical
hyperbolic G{\aa}rding--Dirichlet operators satisfying the central-ray
hypothesis \cite{HarveyLawsonMajorization,MR4589710,MR4965192}.
For the exponent $d/n$, our criterion also applies to nonhomogeneous
operators and to a G{\aa}rding class extending beyond hyperbolicity.
Even within the hyperbolic class, positive determinant capacity does
not require the central-ray hypothesis or attainment of the capacity;
see Remark~\ref{rem:comparison-HL}. Pham has extended determinant
majorization to hyperbolic polynomials on Euclidean Jordan algebras
\cite{Pham26}.

We give several applications illustrating the two main features of the
relative estimate.  First, for the $J$-equation, a singular strict barrier
naturally defines a viscosity $\mathcal D$-subsolution.  The determinant
reduction of Section~\ref{sec:relative-Linfty} then gives a uniform lower bound for smooth
perturbations relative to the barrier.  We formulate this as
Proposition~\ref{prop:J-relative}; this is the form suited to the semistable $J$-equation in our companion work.  Example~\ref{ex:J-fourfold-null-curve} illustrates how a singular
barrier yields uniform local $L^\infty$ bounds away from a null
curve on a semistable fourfold, without symmetry assumptions
on the background metrics.

We also consider critical degenerations of the deformed Hermitian--Yang--Mills (dHYM) equation,
also known as the Leung--Yau--Zaslow (LYZ) equation,
and an inverse $\sigma_2$ equation.  In both cases
the natural background forms converge to semipositive big forms, while
a direct algebraic calculation gives a uniform determinant bound.
The relative estimate therefore remains uniform at the boundary of
the smooth solvability region, even though the corresponding
second-order estimates may degenerate.

Following \cite{fang2026liouvillerigidityrealcomplex}, this is the second paper in our program
on fully nonlinear geometric PDEs based on the G{\aa}rding polynomial
framework developed in \cite{Fang-MaGard,fang2026idealgaardingpolynomials}.
Subsequent work will address semistable $J$-equations and the
$C^2$ and higher-order estimates for related fully nonlinear equations.

We conclude with a brief outline of the paper.
Section~\ref{sec:admissible-equations} introduces admissible pairs, det-ellipticity, and
$\mathcal D$-subsolutions. Section~\ref{sec:relative-Linfty} proves the relative $L^\infty$-estimate in big
classes. Section~\ref{sec:K-Garding} develops $\mathcal K$-G{\aa}rding polynomials and
G{\aa}rding elliptic operators, while Section~\ref{sec:det-ellipticityanddet} establishes the
capacity, Newton-polytope, and slope-semistability characterizations.
Section~\ref{sec:examples-applications} gives applications to the $J$-equation, dHYM/LYZ equations, and an inverse $\sigma_2$ equation.

`\paragraph{\textbf{Declaration of AI assistance.}}
All ideas and mathematical arguments are those of the authors.
Multiple AI services were used to assist in auditing the arguments
and editing the manuscript. The authors take full responsibility
for the mathematical content and the final text.`

\section{Admissible equations and \texorpdfstring{$\mathcal D$}{D}-subsolutions}
\label{sec:admissible-equations}

Let $(X,\rho)$ be a connected compact K\"ahler manifold of complex
dimension $n$. Let $\mathcal Q=\Lambda_{\mathbb R}^{1,1}(X)$, and
denote its fiber at $x\in X$ by $\mathcal Q_x$. The metric $\rho$
identifies $\mathcal Q_x$ with the space of $n\times n$ Hermitian
matrices. We denote by $\mathcal Q_x^+$ and $\mathcal Q_x^{++}$ the
sets of positive semidefinite and positive definite forms,
respectively. For $A\in\mathcal Q_x$ with $\rho$-eigenvalues
$\lambda_1,\ldots,\lambda_n$, we write
$$
\det_\rho A=\frac{A^n}{\rho^n}=\prod_{j=1}^n\lambda_j,
\qquad
\det_{\rho,+}A=\prod_{j=1}^n\max\{\lambda_j,0\}.
$$
We write $\|A\|_\rho=\max_j|\lambda_j|$ for the operator norm with
respect to $\rho$. Thus $\det_{\rho,+}A=\det_\rho A$ when $A\geq0$
and vanishes otherwise.

\subsection{Admissible pairs, quasi-subharmonicity and det-ellipticity}
\label{subsec:det-ellipticity}

\begin{defn}
	\label{def:admissible-pair}
	Let $\mathcal A\subset\mathcal Q$ be open and
	$F\in C^\infty(\mathcal A,\mathbb R)$. We call $(F,\mathcal A)$ an
	\emph{admissible pair} if, for every $x\in X$, the following hold:
	\begin{enumerate}[label=\textnormal{(A\arabic*)}, ref=\textnormal{(A\arabic*)}]
		\item\label{admp:A1}
		$\mathcal A_x:=\mathcal A\cap\mathcal Q_x$ is nonempty and
		$\mathcal A_x+\mathcal Q_x^+\subset\mathcal A_x$.
		\item\label{admp:A2}
		$F_x(A+P)>F_x(A)$ for every $A\in\mathcal A_x$ and
		$P\in\mathcal Q_x^{++}$.
		\item\label{admp:A3}
		For every $A_0\in\partial\mathcal A_x$,
		$$
		\limsup_{\substack{A\in\mathcal A_x\\ A\to A_0}}F_x(A)\leq0.
		$$
	\end{enumerate}
\end{defn}

Given an admissible pair $(F,\mathcal A)$, a smooth closed real
$(1,1)$-form $\omega$, and $\psi\in C^\infty(X)$ with $\psi>0$, we
consider the equation
\begin{equation}
	\label{eq:main-equation}
	\left\{
	\begin{aligned}
		F_x\bigl(\omega_u(x)\bigr)&=\psi(x),\qquad x\in X,\\
		\omega_u(x)&=\omega(x)+\ddc u(x)\in\mathcal A_x,\\
		\sup_Xu&=0.
	\end{aligned}
	\right.
\end{equation}

Here and in what follows, we write $\omega_v=\omega+\ddc v$ for any
potential $v$.



We introduce two structural conditions on admissible pairs. The first is
standard, whereas the second is one of the main structural hypotheses of
the paper.

\begin{defn}[Uniform quasi-subharmonicity]
	\label{def:QSH}
	Let $(F,\mathcal A)$ be an admissible pair. We say that $\mathcal A$
	satisfies the \emph{uniform quasi-subharmonicity condition}, abbreviated
	\emph{(QSH)}, if there exists a constant $b_{\mathcal A}\geq0$ such that
	\begin{equation}
		\label{eq:QSH}
		\tr_\rho A\geq -b_{\mathcal A}
	\end{equation}
	for every $x\in X$ and every $A\in\mathcal A_x$.
\end{defn}

The uniform quasi-subharmonicity condition is used through the
following lemma.
\begin{lem}
	\label{lem:admissible-L1}
	Let $(F,\mathcal A)$ be an admissible pair satisfying QSH,
	and let $u\in C^2(X)$ solve \eqref{eq:main-equation}. Then
	$$
	\int_X(-u)\,\rho^n
	\leq C(X,\rho)
	\left(b_{\mathcal A}
	+\|\operatorname{tr}_\rho\omega\|_{L^\infty(X)}\right).
	$$
\end{lem}

\begin{proof}
	By \eqref{eq:QSH},
	$$
	\Delta_\rho u
	=\operatorname{tr}_\rho\omega_u-\operatorname{tr}_\rho\omega
	\geq-b_{\mathcal A}
	-\|\operatorname{tr}_\rho\omega\|_{L^\infty(X)}.
	$$
	The conclusion follows from $\sup_Xu=0$ and the Green-function
	estimate on $(X,\rho)$.
\end{proof}


\begin{defn}
	\label{def:det-elliptic}
	An admissible pair $(F,\mathcal A)$ is called
	\emph{determinantally elliptic}, or \emph{det-elliptic}, if there
	exist constants $\eta,\beta>0$ such that
	\begin{equation}
		\label{eq:det-elliptic}
		F_x(A+P)-F_x(A)\geq\eta\,\det_\rho(P)^\beta
	\end{equation}
	for every $x\in X$, $A\in\mathcal A_x$, and
	$P\in\mathcal Q_x^{++}$.
\end{defn}

Condition~\eqref{eq:det-elliptic} is a quantitative strengthening of
condition~\ref{admp:A2}. By continuity, it also holds for $P\geq0$.
Analytically, det-ellipticity provides a systematic mechanism for producing
$\mathcal D$-subsolutions and hence relative $L^\infty$ estimates.

\subsection{\texorpdfstring{$\mathcal C$}{C}-subsolutions and
	\texorpdfstring{$\mathcal D$}{D}-subsolutions}
\label{subsec:D-subsolutions}

We first recall the definition of viscosity solutions for \eqref{eq:main-equation}.



\begin{defn}
	\label{def:viscosity-solutions}
	Let $M\subset X$ be open. A real-valued function
	$\underline u\in\operatorname{USC}(M)$ is a \emph{viscosity
		subsolution} of \eqref{eq:main-equation} on $M$ if every
	$q\in C^2$ touching $\underline u$ from above at $x\in M$ satisfies
	$$
	\omega(x)+\ddc q(x)\in\mathcal A_x,
	\qquad F_x\bigl(\omega(x)+\ddc q(x)\bigr)\geq\psi(x).
	$$
	A real-valued function $\overline u\in\operatorname{LSC}(M)$ is a
	\emph{viscosity supersolution} if every $q\in C^2$ touching
	$\overline u$ from below at $x\in M$ satisfies
	$$
	\omega(x)+\ddc q(x)\in\mathcal A_x
	\quad\Longrightarrow\quad
	F_x\bigl(\omega(x)+\ddc q(x)\bigr)\leq\psi(x).
	$$
	A function $u\in C(M)$ is a \emph{viscosity solution} if it
	satisfies both conditions.
\end{defn}

When a potential is smooth and admissible, the viscosity notions of subsolution and supersolution agree with the pointwise ones.



With the basic solution concepts in place, we now introduce the
structural $\mathcal C$- and $\mathcal D$-subsolutions that drive the
main estimates. We use a shifted form of the $\mathcal C$-subsolution
condition.

\begin{defn}
	\label{def:C-subsolution}
	Let $\underline u\in C^\infty(X)$, let $\chi$ be a smooth closed real
	$(1,1)$-form, and let $R>0$. We call the triple
	$(\omega_{\underline u},\chi,R)$ a \emph{$\mathcal C$-subsolution}
	of \eqref{eq:main-equation} if, for every $x\in X$ and
	$P\in\mathcal Q_x^+$ satisfying
	$$
	\left\{
	\begin{aligned}
		\omega_{\underline u}(x)-\chi(x)+P&\in\mathcal A_x,\\
		F_x\bigl(\omega_{\underline u}(x)-\chi(x)+P\bigr)&\leq\psi(x),
	\end{aligned}
	\right.
	$$
	one has $\|P\|_\rho\leq R$.
	When $\chi$ and $R$ are fixed, we also call
	$\omega_{\underline u}$, or simply $\underline u$, a
	$\mathcal C$-subsolution.
\end{defn}

For $\chi=\delta\rho$ with $\delta>0$, this is the uniform shifted
form of Sz\'ekelyhidi's $\mathcal C$-subsolution condition
\cite{MR3807322}; see also
\cite[Definition~1.1]{guo2024uniformlinftyestimatessubsolutions}.
Related subsolution conditions were developed by Guan; see
\cite{MR3284698,MR4727578}.  In the spectral setting
$$
\mathcal A_x=\lambda_\rho^{-1}(\Gamma),
\qquad F_x=f\circ\lambda_\rho,
$$
where $\lambda_\rho$ is the eigenvalue map, the conclusion
$\|P\|_\rho\leq R$ is equivalent to $P\leq R\rho(x)$ and bounds
every eigenvalue of the positive increment. The definition is given without reference to eigenvalues.

The following definition extends the determinant sublevel condition
of Sui--Sun~\cite[condition~(3)]{MR4567566} from spectral equations
to general admissible pairs.

\begin{defn}
	\label{def:D-subsolution}
	Let $\underline u\in C^\infty(X)$, let $\chi$ be a smooth closed real
	$(1,1)$-form, and let $\gamma\in C(X,\mathbb R)$. We call the triple
	$(\omega_{\underline u},\chi,\gamma)$ a \emph{determinantal
		subsolution}, or a \emph{$\mathcal D$-subsolution}, of
	\eqref{eq:main-equation} if, for every $x\in X$ and
	$P\in\mathcal Q_x^+$ satisfying
	\begin{equation}
		\label{eq:D-subsolution-test}
		\left\{
		\begin{aligned}
			\omega_{\underline u}(x)-\chi(x)+P&\in\mathcal A_x,\\
			F_x\bigl(\omega_{\underline u}(x)-\chi(x)+P\bigr)&\leq\psi(x),
		\end{aligned}
		\right.
	\end{equation}
	one has $\det_{\rho(x)}P\leq e^{\gamma(x)}$.
	When $\chi$ and $\gamma$ are fixed, we also call
	$\omega_{\underline u}$, or simply $\underline u$, a
	$\mathcal D$-subsolution.
\end{defn}

Both definitions are pointwise and apply on open subsets of $X$. The two notions measure the same admissible sublevel set in
different ways. A $\mathcal C$-subsolution
$(\omega_{\underline u},\chi,R)$ gives a $\mathcal D$-subsolution
$(\omega_{\underline u},\chi,n\log R)$. The determinant bound
alone does not bound the spectrum.

As in Sz\'ekelyhidi's use of $\mathcal C$-subsolutions, a
$\mathcal D$-subsolution suffices for the ABP argument under a
uniform positivity assumption; see Appendix~\ref{ABP}. 

Det-ellipticity produces $\mathcal D$-subsolutions directly from
the operator, with an explicit density.
\begin{prop}
	\label{prop:det-elliptic-D-subsolution}
	Let $\underline u$ be smooth. Suppose that $(F,\mathcal A)$ is
	det-elliptic and that $\omega_{\underline u}-\chi\in\mathcal A$. Then every $P\geq0$ satisfying
	\eqref{eq:D-subsolution-test} obeys
	$$
	\det_\rho P\leq d(x),
	\qquad
	d(x):=\left[
	\eta^{-1}\bigl(\psi(x)-F_x(\omega_{\underline u}(x)-\chi(x))\bigr)_+
	\right]^{1/\beta}.
	$$
	Consequently, $(\omega_{\underline u},\chi,\gamma)$ is a
	$\mathcal D$-subsolution for every continuous $\gamma$ with
	$e^\gamma\geq d$.
\end{prop}

\begin{proof}
	Set $A=\omega_{\underline u}(x)-\chi(x)$. By \eqref{eq:det-elliptic},
	$$
	\eta\,\det_\rho(P)^\beta
	\leq F_x(A+P)-F_x(A)
	\leq\psi(x)-F_x(A).
	$$
	This proves the assertion.
\end{proof}


\subsection{Weak and singular \texorpdfstring{$\mathcal D$}{D}-subsolutions}
\label{subsec:weak-D-subsolutions}

We extend Definition~\ref{def:D-subsolution} to nonsmooth
potentials.

\begin{defn}
	\label{def:viscosity-D-subsolution}
	Let $M\subset X$ be open, let  $\chi$ be a smooth closed real $(1, 1)$-form, let $\underline u\in\operatorname{USC}(M)$
	be real-valued, and let $\gamma\in C(M,\mathbb R)$. We call the triple
	$(\omega_{\underline u},\chi,\gamma)$ a \emph{viscosity
		$\mathcal D$-subsolution} of \eqref{eq:main-equation} on $M$ if,
	for every $\varphi\in C^2$ touching $\underline u$ from above at
	$x\in M$ and every $P\in\mathcal Q_x^+$ satisfying
	\begin{equation}
		\label{eq:viscosity-D-test}
		\left\{
		\begin{aligned}
			\omega(x)+\ddc\varphi(x)-\chi(x)+P&\in\mathcal A_x,\\
			F_x\bigl(\omega(x)+\ddc\varphi(x)-\chi(x)+P\bigr)&\leq\psi(x),
		\end{aligned}
		\right.
	\end{equation}
	one has $\det_{\rho(x)}P\leq e^{\gamma(x)}$.
	When $\chi$ and $\gamma$ are fixed, we also call
	$\omega_{\underline u}$, or simply $\underline u$, a viscosity
	$\mathcal D$-subsolution.
\end{defn}

For nonsmooth potentials, $\mathcal D$-subsolution will always
mean viscosity $\mathcal D$-subsolution.  We suppress the fixed data also when discussing
regularizability and singularities.

\begin{prop}
	\label{prop:D-consistency-stability}
	For smooth $\underline u$, the pointwise and viscosity
	$\mathcal D$-subsolution definitions for
	$(\omega_{\underline u},\chi,\gamma)$ are equivalent.
	Moreover, suppose that $\underline u_\nu$ are continuous viscosity
	$\mathcal D$-subsolutions with data $(\chi,\gamma_\nu)$ and that
	$\underline u_\nu\to\underline u$ and $\gamma_\nu\to\gamma$
	locally uniformly. Then $\underline u$ is a viscosity
	$\mathcal D$-subsolution with data $(\chi,\gamma)$.
\end{prop}

\begin{proof}
	The first assertion is elementary: testing a smooth potential with
	itself gives one implication, and the converse follows by applying
	the pointwise condition to $P+H$, where
	$H=\ddc\varphi(x)-\ddc\underline u(x)\ge0$, together with the
	monotonicity of $\det_\rho$.
	
	For stability, suppose the bound fails for a test $\varphi$ touching
	$\underline u$ from above at $x_0$: for some $P\geq0$ satisfying
	\eqref{eq:viscosity-D-test},
	$\det_{\rho(x_0)}P>e^{\gamma(x_0)}$. Since $\gamma(x_0)\in\mathbb R$,
	$e^{\gamma(x_0)}>0$, hence $P>0$. For sufficiently small
	$\varepsilon>0$,
	$$
	P-\varepsilon\rho(x_0)>0,\qquad
	\det_{\rho(x_0)}\bigl(P-\varepsilon\rho(x_0)\bigr)>e^{\gamma(x_0)}.
	$$
	By openness of $\mathcal A$, the form
	$\omega(x_0)+\ddc\varphi(x_0)-\chi(x_0)+P-\varepsilon\rho(x_0)$
	remains admissible, and condition~\ref{admp:A2} makes the inequality
	strict: $F_{x_0}(\,\cdot\,)<\psi(x_0)$. Add a positive quartic term
	to $\varphi$ in local coordinates to make the contact strict
	without changing its Hessian, extend $P$ smoothly nearby, and use
	local uniform convergence to obtain upper contacts of
	$\underline u_\nu$ at points $x_\nu\to x_0$. For large $\nu$ the
	strict admissibility and the gap persist at $x_\nu$, hence
	$$
	\det_{\rho(x_\nu)}\bigl(P(x_\nu)-\varepsilon\rho(x_\nu)\bigr)
	\leq e^{\gamma_\nu(x_\nu)}.
	$$
	Letting $\nu\to\infty$ contradicts the strict determinant gap,
	proving stability.
\end{proof}

Motivated by Proposition~\ref{prop:D-consistency-stability}, we
introduce the following definition.

\begin{defn}
	\label{def:regularizable-D-subsolution}
	Let $\underline u$ be a continuous viscosity $\mathcal D$-subsolution
	on $M$, with $\chi$ and $\gamma$ fixed. We call $\underline u$
	\emph{locally regularizable} if every point of $M$ has a neighborhood
	$V\Subset M$ on which there are smooth functions $\underline u_\nu$
	and $\gamma_\nu$ such that
	$$
	\underline u_\nu\longrightarrow\underline u,
	\qquad \gamma_\nu\longrightarrow\gamma
	\quad\hbox{locally uniformly on }V,
	$$
	and $(\omega_{\underline u_\nu},\chi,\gamma_\nu)$ is a smooth
	$\mathcal D$-subsolution on $V$, with the background $\omega$ and
	the equation unchanged.
	
	If $M=X\setminus Z$ for a closed set $Z$ and
	$\underline u(x)\to-\infty$ as $x\to Z$, we call $\underline u$ a
	\emph{singular $\mathcal D$-subsolution}. It is singular locally
	regularizable if it also has the preceding approximation property
	on $X\setminus Z$.
\end{defn}

\section{Relative \texorpdfstring{$L^\infty$}{L-infinity} estimates}
\label{sec:relative-Linfty}

\subsection{Determinant reduction and the main estimate}
\label{subsec:viscosity-relative-Linfty}

We retain the fixed equation and notation of
Section~\ref{sec:admissible-equations}. The scalar inequality used
in the proof is the following.

\begin{defn}
	\label{def:viscosity-determinant-bound}
	Let $M\subset X$ be open, let $\chi$ be a smooth closed real
	$(1,1)$-form, and let $\gamma\in C(M,\mathbb R\cup\{\pm\infty\})$,
	continuous in the extended sense, with the conventions
	$e^{-\infty}=0$ and $e^{+\infty}=+\infty$. We say that a
	real-valued function $U\in\operatorname{LSC}(M)$ satisfies the
	\emph{$(\chi,\gamma)$-determinant bound in the viscosity sense}
	if every $q\in C^2$ touching $U$ from below at $x\in M$ satisfies
	\begin{equation}
		\label{eq:viscosity-determinant-bound}
		\det_{\rho(x),+}\bigl(\chi(x)+\ddc q(x)\bigr)\leq e^{\gamma(x)}.
	\end{equation}
	At points with $\gamma=-\infty$ this reads
	$\det_{\rho(x),+}(\chi(x)+\ddc q(x))=0$; at points with
	$\gamma=+\infty$ the inequality is void.
\end{defn}

When $\chi$ and $\gamma$ are fixed, we call \eqref{eq:viscosity-determinant-bound} the
\emph{viscosity determinant bound}. Note that only lower tests with $\chi+\ddc q\geq0$
impose a condition.

\begin{prop}
	\label{prop:determinant-reduction}
	Let $M\subset X$ be open, fix $\chi$ as above and a finite $\gamma\in C(M,\mathbb R)$, and let
	$\overline u$ be a viscosity supersolution of
	\eqref{eq:main-equation} on $M$. If $\underline u$ is a locally
	regularizable $\mathcal D$-subsolution, then
	$U=\overline u-\underline u$ satisfies the viscosity determinant
	bound on $M$.
	
	If $\overline u\in C^2(M)$ is an admissible supersolution, the same
	conclusion holds for every continuous viscosity
	$\mathcal D$-subsolution $\underline u$, without assuming local
	regularizability.
\end{prop}

\begin{proof}
	Let $q$ touch $\overline u-\underline u$ from below at $x_0$.
	We need only consider $P:=\chi(x_0)+\ddc q(x_0)>0$. In local coordinates, set
	$q_\varepsilon(z)=q(z)-\varepsilon|z-x_0|^2$. For the paired local
	regularizations $(\underline u_\nu,\gamma_\nu)$ supplied by
	Definition~\ref{def:regularizable-D-subsolution}, the functions
	$\overline u-\underline u_\nu-q_\varepsilon$ have local minima at points
	$x_\nu\to x_0$. Set
	$$
	P_\nu:=\chi(x_\nu)+\ddc q_\varepsilon(x_\nu),
	\qquad
	A_\nu:=\omega(x_\nu)+\ddc(\underline u_\nu+q_\varepsilon)(x_\nu).
	$$
	For large $\nu$, we have $P_\nu>0$ and
	$A_\nu=\omega_{\underline u_\nu}(x_\nu)-\chi(x_\nu)+P_\nu$.
	
	If $A_\nu\in\mathcal A_{x_\nu}$, the supersolution inequality and
	the defining property of $\underline u_\nu$ give
	$\det_{\rho(x_\nu)}P_\nu\leq e^{\gamma_\nu(x_\nu)}$.
	If $A_\nu\notin\mathcal A_{x_\nu}$, put
	$$
	t_\nu:=\inf\{t\geq0:A_\nu+t\rho(x_\nu)\in\mathcal A_{x_\nu}\}.
	$$
	The set is nonempty: any fixed element of the nonempty fiber
	$\mathcal A_{x_\nu}$ is dominated by $A_\nu+t\rho(x_\nu)$ for
	large $t$, and condition~\ref{admp:A1} applies. It is an upper
	interval, and its endpoint lies on $\partial\mathcal A_{x_\nu}$.
	By condition~\ref{admp:A3} and $\psi(x_\nu)>0$, there is
	$t>t_\nu$ such that
	$$
	A_\nu+t\rho(x_\nu)\in\mathcal A_{x_\nu},
	\qquad
	F_{x_\nu}\bigl(A_\nu+t\rho(x_\nu)\bigr)\leq\psi(x_\nu).
	$$
	Applying the smooth $\mathcal D$-subsolution property with the
	positive increment $P_\nu+t\rho(x_\nu)$ yields
	$$
	\det_{\rho(x_\nu)}P_\nu
	\leq\det_{\rho(x_\nu)}\bigl(P_\nu+t\rho(x_\nu)\bigr)
	\leq e^{\gamma_\nu(x_\nu)}.
	$$
	Letting $\nu\to\infty$ and then $\varepsilon\downarrow0$, we obtain
	$$
	\det_{\rho(x_0)}P\leq e^{\gamma(x_0)}.
	$$
	
	For the last assertion, $\overline u-q$ is itself a $C^2$ upper
	test for $\underline u$ at $x_0$. At that point,
	$$
	\omega+\ddc(\overline u-q)-\chi+P
	=\omega_{\overline u}\in\mathcal A,
	\qquad F(\omega_{\overline u})\leq\psi.
	$$
	Definition~\ref{def:viscosity-D-subsolution} gives
	\eqref{eq:viscosity-determinant-bound} directly.
\end{proof}

\begin{rem}
	Only conditions~\ref{admp:A1} and~\ref{admp:A3} enter this
	reduction; the strict monotonicity condition~\ref{admp:A2} is
	not used.
\end{rem}

We now fix the notation used in the rest of the section. Let
$\chi$ be a smooth closed real $(1,1)$-form. When $[\chi]$ is big,
its volume $V_\chi>0$, its canonical envelope $W_\chi$, and the
determinant entropy $\operatorname{Ent}_{\mathcal D,q}(\chi,\gamma,\rho)$
are those defined in \eqref{eq:intro-volume},
\eqref{eq:intro-envelope}, and \eqref{eq:intro-D-entropy},
respectively. We write $\operatorname{PSH}_0(X,\chi)$ for the
elements of $\operatorname{PSH}(X,\chi)$ with supremum zero.

\begin{thm}
	\label{thm:determinant-relative-Linfty}
	Let $(X, \rho)$ be a compact K\"{a}hler manifold, let
	$\chi$ be a smooth closed real $(1,1)$-form with $[\chi]$ big,
	and let $Z\subset X$ be closed and proper. Suppose that
	$\gamma\in C(X\setminus Z,[-\infty,+\infty))$ in the extended sense,
	that $e^\gamma>0$ almost everywhere on $X\setminus Z$, and that
	its zero extension belongs to $L^{p_0}(X,\rho^n)$ for some $p_0>1$.
	Let $U\in C(X,\mathbb R\cup\{+\infty\})$ satisfy
	$\{U=+\infty\}=Z$ and the viscosity determinant bound
	\eqref{eq:viscosity-determinant-bound} on $X\setminus Z$.
	Assume that, for some $M>0$,
	\begin{equation}
		\label{eq:thm-L1-hypothesis}
		\int_X(W_\chi-U)_+\,\rho^n\leq M,
	\end{equation}
	and that one of the following cases holds:
	\begin{enumerate}[label=\textnormal{(\alph*)}]
		\item\label{enu:thm-exp}
		\emph{Exponential case.} For some $\alpha>0$ and $q>n$,
		\[
		\operatorname{Ent}_{\mathcal D,q}(\chi,\gamma,\rho)\leq M,
		\qquad
		\sup_{\varphi\in\operatorname{PSH}_0(X,\chi)}
		\int_Xe^{-\alpha\varphi}\rho^n\leq M.
		\]
		\item\label{enu:thm-poly}
		\emph{Power case.} For some $p>1$ and $q>np'$, where $p'=p/(p-1)$,
		\[
		\int_X e^{p(\gamma-\log V_\chi)}\,\rho^n\leq M,
		\qquad
		\sup_{\varphi\in\operatorname{PSH}_0(X,\chi)}
		\int_X(-\varphi)^q\rho^n\leq M.
		\]
	\end{enumerate}
	Then $\sup_X(W_\chi-U)\leq C$. One may take
	$C=C(n,q,\alpha,M)$ in case~\ref{enu:thm-exp}, and
	$C=C(p,q,n,M,\int_X \rho^n)$ in case~\ref{enu:thm-poly}.
\end{thm}

Sui--Sun established relative $L^\infty$ estimates for smooth
Hessian quotient equations with nef and big auxiliary
classes~\cite[Theorem~1.2]{MR4567566}, as well as quantitative
stability for complex $k$-Hessian
equations~\cite[Theorem~1.5]{MR4567566}.
Our contribution here concerns the relative estimate in the
general setting of admissible pairs, without a spectral
assumption. Theorem~\ref{thm:determinant-relative-Linfty}
allows arbitrary big auxiliary classes and determinant
viscosity supersolutions, with constants controlled, in the
entropy case, by determinant entropy and an unweighted
relative $L^1$ bound.

The density $e^\gamma/V_\chi$ is finite and continuous on
$X\setminus Z$, where it may vanish. Infinite densities are
treated in Theorem~\ref{thm:infinite-density-relative-Linfty}.

An a priori bound for the weighted energy $A_0$ defined in
\eqref{eq:level-set-quantities} can replace the unweighted
$L^1$ hypothesis~\eqref{eq:thm-L1-hypothesis}: in this case,
use part~\ref{enu:unified-growth} in place of
part~\ref{enu:unified-moments}. The weighted moment and mass
bounds are still required.

\subsection{Technical tools}
\label{subsec:big-class-scalar-estimate}

Write
\[
\Omega_\chi:=\operatorname{Amp}([\chi])
\]
for the ample locus of $[\chi]$. A potential
$\varphi\in\operatorname{PSH}(X,\chi)$ has minimal singularities
precisely when its difference from $W_\chi$ is bounded on $X$.

We collect the pluripotential results used below. The existence
of a K\"ahler current with analytic singularities is due to
Boucksom~\cite[Theorem~3.17]{MR2050205}, and the regularity of
the envelope is due to
Berman--Demailly~\cite[Theorem~1.4]{MR2884031}.
The solvability of the Monge--Amp\`ere equation is due to
Boucksom--Eyssidieux--Guedj--Zeriahi
\cite[Theorems~A and B]{MR2746347}, while the local H\"older
continuity of its solutions on the ample locus is due to
Demailly--Dinew--Guedj--Pham--Ko\l odziej--Zeriahi
\cite[Theorem~D]{MR3191972}.

\begin{prop}
	\label{prop:big-class-inputs}
	Let $(X,\rho)$ be a compact K\"ahler manifold, and let
	$\chi$ be a smooth closed real $(1,1)$-form with $[\chi]$
	big.
	
	The envelope $W_\chi$ has minimal singularities and is
	continuous on $\Omega_\chi$. There exists
	$\zeta_0\in\operatorname{PSH}(X,\chi)$ with analytic
	singularities, smooth precisely on $\Omega_\chi$, such that
	\begin{equation}
		\label{eq:strict-big-class-potential}
		\chi+\ddc\zeta_0\geq\delta_0\rho
	\end{equation}
	for some $\delta_0>0$.
	
	If $g\geq0$ belongs to $L^p(X,\rho^n)$ for some $p>1$ and
	$\int_Xg\,\rho^n=V_\chi$, there exists a unique
	$\varphi\in\operatorname{PSH}_0(X,\chi)$ satisfying
	$$
	\left\langle(\chi+\ddc\varphi)^n\right\rangle
	=g\rho^n,
	$$
	where angle brackets denote the non-pluripolar product.
	This solution has minimal singularities and is locally
	H\"older continuous on $\Omega_\chi$. In particular,
	there exists $C_g<\infty$ such that
	$0\leq W_\chi-\varphi\leq C_g$ on $X$.
\end{prop}

Here and below, angle brackets denote non-pluripolar products of
closed positive $(1,1)$-currents. On compact K\"ahler manifolds,
these products are well defined and closed
\cite[Proposition~1.6, Theorem~1.8]{MR2746347}; on an open set where
the potentials are locally bounded, they agree with the
Bedford--Taylor products. The volume is defined in
\cite[Definition~1.17]{MR2746347} and is independent of the
representative by \cite[Theorem~1.16]{MR2746347}; the ample locus is
defined in \cite[Definition~1.13]{MR2746347}.

We also need uniform exponential integrability. We record its
uniformity in the background form for later use.

\begin{lem}
	\label{lem:uniform-alpha-integrability}
	If $\chi\leq K\rho$ for some $K>0$, there are constants
	$\alpha,C_\alpha>0$, depending only on $(X,\rho,K)$, such that
	\begin{equation}
		\label{eq:uniform-alpha-integrability}
		\sup_{\varphi\in\operatorname{PSH}_0(X,\chi)}
		\int_Xe^{-\alpha\varphi}\rho^n\leq C_\alpha.
	\end{equation}
\end{lem}
The assertion is the uniform form of Skoda's integrability theorem,
applied to the rescaled normalized family
$\{\varphi/K:\varphi\in\operatorname{PSH}_0(X,\chi)\}\subset\operatorname{PSH}_0(X,\rho)$;
see \cite[proof of Proposition~4.3]{MR2746347}.

We use the pluripotential--viscosity equivalence and local comparison
principle of Guedj--Zeriahi
\cite[Proposition~6.21, 6.25 and Theorem~6.31]{MR3617346}.
The comparison theorem applies to
$(\ddc v)^n=e^{\varepsilon v}f\,\beta^n$
with $\varepsilon>0$ and continuous $f\geq0$.
Here $\beta=\ddc|z|^2$, and $(\ddc b)^n_+$ denotes the
positive-determinant operator at lower test functions, with the
convention of Section~\ref{sec:admissible-equations}.

\begin{lem}
	\label{lem:strict-factor-comparison}
	Let $B\Subset\mathbb C^n$ be a ball, let
	$g\in C(\overline B)$ be nonnegative, and let $c>1$.
	Suppose that $a\in\operatorname{PSH}(B)$ and
	$b\in\operatorname{LSC}(B)$ are bounded and satisfy
	$$
	(\ddc a)^n_{\mathrm{BT}}\geq cg\,\beta^n,
	\qquad
	(\ddc b)^n_+\leq g\,\beta^n,
	$$
	where the second inequality is understood in the viscosity
	sense. Then
	$$
	\sup_B(a-b)
	\leq\limsup_{z\to\partial B}(a-b)(z).
	$$
\end{lem}

\begin{proof}
After adding a constant to $b$, assume that
$\limsup_{z\to\partial B}(a-b)(z)\leq0$.
Set $A=\inf_B b$ and choose $\varepsilon>0$ such that
\[
e^{\varepsilon(\sup_B a-A)}<c.
\]
Since
\[
e^{\varepsilon(a-A)}<c,
\qquad
e^{\varepsilon(b-A)}\geq1,
\]
the pluripotential--viscosity equivalence shows that $a$ is
a viscosity subsolution, while $b$ is a viscosity supersolution, of
\[
(\ddc v)^n=e^{\varepsilon(v-A)}g\,\beta^n.
\]
The local comparison principle
\cite[Theorem~6.31]{MR3617346}, applied with the continuous
nonnegative density $e^{-\varepsilon A}g$, gives $a\leq b$.
\end{proof}

\subsection{The auxiliary Monge--Amp\`ere comparison}
\label{subsec:auxiliary-MA-comparison}

For the proof of Theorem~\ref{thm:determinant-relative-Linfty}, set
\begin{equation}
	\label{eq:singular-density-entropy}
	f:=\frac{e^\gamma}{V_\chi},
	\qquad
	d\mu:=f\,\rho^n,
	\qquad
	h:=(W_\chi-U)_+,
\end{equation}
where $h=0$ on $Z$ (we interpret $W_\chi-U=-\infty$ on $Z$). For $s\geq0$, define
\begin{equation}
	\label{eq:level-set-quantities}
	X_s=\{h>s\},
	\qquad
	A_s=\int_{X_s}(h-s)\,d\mu,
	\qquad
	\Phi(s)=\mu(X_s).
\end{equation}
The extension of $U$ is bounded below on $X$, so $h$ is bounded
for each fixed $U$. 

For $k\geq1$, let $\tau_k(t)=t_++1/k$, with
$\tau_k(-\infty)=1/k$, and set
\begin{equation}
	\label{eq:auxiliary-normalization}
	A_{s,k}=\int_X\tau_k(W_\chi-U-s)\,d\mu
	=A_s+\frac{\mu(X)}k.
\end{equation}
Since $X\setminus Z$ is nonempty and $e^\gamma>0$ almost everywhere on $X \setminus Z$,
$A_{s,k}>0$. Consider
\begin{equation}
	\label{eq:auxiliary-MA-extended}
	\left\langle(\chi+\ddc\phi_{s,k})^n\right\rangle
	=\frac{\tau_k(W_\chi-U-s)}{A_{s,k}}e^\gamma\rho^n,
	\qquad
	\sup_X\phi_{s,k}=0.
\end{equation}
Its right-hand side has mass $V_\chi$, since $e^\gamma=V_\chi f$ and
$\int_X\tau_k(W_\chi-U-s)\,d\mu=A_{s,k}$. Its density lies in
$L^{p_0}$: $\tau_k(W_\chi-U-s)$ is bounded and the zero extension of
$e^\gamma$ belongs to $L^{p_0}$. By
Proposition~\ref{prop:big-class-inputs}, there is a full-mass solution
with minimal singularities, continuous on $\Omega_\chi$, and
\begin{equation}
	\label{eq:auxiliary-minimal-singularities}
	0\leq W_\chi-\phi_{s,k}\leq C_{s,k}
	\quad\hbox{on }\Omega_\chi.
\end{equation}

Fix $s$ and $k$, and set
\begin{equation}
	\label{eq:epsilon-Lambda}
	r=\frac n{n+1},
	\qquad
	\varepsilon_{s,k}=r^{-r}A_{s,k}^{1/(n+1)},
	\qquad
	\Lambda_{s,k}=rA_{s,k}.
\end{equation}
These choices give
$\Lambda_{s,k}^{1-r}=r\varepsilon_{s,k}$,
$r^n\varepsilon_{s,k}^{n+1}=A_{s,k}$, and $n(1-r)=r$.
The comparison function is
\begin{equation}
	\label{eq:comparison-function}
	v=W_\chi-s-\varepsilon_{s,k} T^r,
	\qquad
	T=W_\chi-\phi_{s,k}+\Lambda_{s,k},
	\quad\hbox{on }\Omega_\chi.
\end{equation}

\begin{lem}

\label{lem:comparison-function-psh}

The function $v$ extends to a $\chi$-plurisubharmonic function
with minimal singularities on $X$, and satisfies

\begin{equation}
\label{eq:comparison-function-lower-bound}
\chi+\ddc v
\geq\frac{r\varepsilon_{s,k}}{T^{1-r}}
(\chi+\ddc\phi_{s,k})
\quad\hbox{on }\Omega_\chi.
\end{equation}

Moreover,
\begin{equation}
\label{eq:comparison-function-MA-lower-bound}
(\chi+\ddc v)^n_{\mathrm{BT}}
\geq
\frac{\tau_k(W_\chi-U-s)}
{\varepsilon_{s,k}T^r}\,e^\gamma\rho^n
\quad\hbox{on }\Omega_\chi.
\end{equation}

\end{lem}

\begin{proof}

On a coordinate ball compactly contained in $\Omega_\chi$, let
$\ddc\theta=\chi$ and regularize the continuous psh functions
$W_\chi+\theta\geq\phi_{s,k}+\theta$ by the same convolution
kernel, obtaining $x_j\geq y_j$. Set
\[
G(x,y)=x-s-\varepsilon_{s,k}(x-y+\Lambda_{s,k})^r,
\qquad V_j=G(x_j,y_j),\qquad \lambda_j=G_y(x_j,y_j).
\]
By \eqref{eq:epsilon-Lambda}, $G$ is convex on $\{x\geq y\}$,
with $0<G_y\leq1$ and $G_x=1-G_y$. Hence
\[
\ddc V_j\geq(1-\lambda_j)\ddc x_j+\lambda_j\ddc y_j
\geq\lambda_j\ddc y_j\geq0,
\qquad
(\ddc V_j)^n\geq\lambda_j^n(\ddc y_j)^n.
\]
The potentials and coefficients converge locally uniformly,
with $V_j\to v+\theta$ and
$\lambda_j\to r\varepsilon_{s,k}/T^{1-r}$.
Passing to the limit proves that $v$ is $\chi$-psh and gives
\eqref{eq:comparison-function-lower-bound}.
Bedford--Taylor continuity and
\eqref{eq:auxiliary-MA-extended}--\eqref{eq:epsilon-Lambda}
give \eqref{eq:comparison-function-MA-lower-bound}.

By \eqref{eq:comparison-function} and
\eqref{eq:auxiliary-minimal-singularities},
$W_\chi-C\leq v\leq W_\chi$ on $\Omega_\chi$.
Choose $\zeta_0$ from \eqref{eq:strict-big-class-potential},
shifted so that $\zeta_0\leq v$. The functions
\[
v_\delta=(1-\delta)v+\delta\zeta_0,\qquad 0<\delta<1,
\]
are $\chi$-psh, bounded above, and tend to $-\infty$ along
$X\setminus\Omega_\chi$, so they extend by $-\infty$ to $X$.
The upper semicontinuous regularization of their increasing
limit as $\delta\downarrow0$ extends $v$.
The bounds $W_\chi-C\leq v\leq W_\chi$ then give minimal singularities.

\end{proof}

\begin{lem}
	\label{lem:comparison-function-below-U}
	The function in \eqref{eq:comparison-function} satisfies $v\leq U$
	on $X$.
\end{lem}

\begin{proof}
	Keep $\zeta_0\leq v$ as in the preceding proof. Then
	$v_\delta\leq v$, $v_\delta=-\infty$ on
	$X\setminus\Omega_\chi$, and
	\begin{equation}
		\label{eq:vdelta-positive-current}
		\chi+\ddc v_\delta
		\geq(1-\delta)\frac{r\varepsilon_{s,k}}{T^{1-r}}
		(\chi+\ddc\phi_{s,k})+\delta\delta_0\rho.
	\end{equation}
	Suppose that $m=\sup_X(v-U)>0$. The extension property of $v$ and
	continuity of $U$ off $Z$ give a point
	$x_1\in\Omega_\chi\setminus Z$ with $(v-U)(x_1)>3m/4$.
	By \eqref{eq:auxiliary-minimal-singularities}, choose $K_{s,k}<\infty$
	such that $T^r\leq K_{s,k}$ on $\Omega_\chi$. We may fix
	$\delta>0$ so small that
	$$
	(v_\delta-U)(x_1)>m/2,
	\qquad
	c_*:=(1-\delta)^n
	\left(1+\frac{m}{4\varepsilon_{s,k} K_{s,k}}\right)>1.
	$$
	The function $v_\delta-U$ tends to $-\infty$ both along
	$X\setminus\Omega_\chi$ and along $Z$. It therefore attains a
	maximum $m_\delta>m/2$ at a point
	$x_\delta\in\Omega_\chi\setminus Z$.
	
	Choose a coordinate ball $B_R$ centered at $x_\delta$ such that
	$$
	\overline{B_R}\Subset\Omega_\chi\setminus Z,
	\qquad
	v_\delta-U>m/4\quad\hbox{on }\overline{B_R}.
	$$
	For $\kappa>0$ small enough that
	$\kappa\ddc|z-x_\delta|^2\leq\delta\delta_0\rho$ on this ball, set
	$$
	w=v_\delta+\kappa(R^2-|z-x_\delta|^2).
	$$
	It follows from \eqref{eq:vdelta-positive-current} that
	\[
	\chi+\ddc w\geq(1-\delta)r\varepsilon_{s,k} T^{r-1}
	(\chi+\ddc\phi_{s,k}).
	\]
	Thus $w$ is locally $\chi$-plurisubharmonic.
Bedford--Taylor multilinearity and positivity, together with
\eqref{eq:comparison-function-MA-lower-bound}, give on $B_R$
\begin{align*}
	(\chi+\ddc w)^n_{\mathrm{BT}}
	&=\Bigl((1-\delta)(\chi+\ddc v)
	+\underbrace{\delta(\chi+\ddc\zeta_0)
	-\kappa\ddc|z-x_\delta|^2}_{\geq0}\Bigr)^n_{\mathrm{BT}}\\
	&\geq(1-\delta)^n(\chi+\ddc v)^n_{\mathrm{BT}}\\
	&\geq(1-\delta)^n
	\frac{\tau_k(W_\chi-U-s)}{\varepsilon_{s,k}T^r}\,
	e^\gamma\rho^n\\
	&\geq(1-\delta)^n
	\left(1+\frac{v-U}{\varepsilon_{s,k}T^r}\right)
	e^\gamma\rho^n
	\geq c_*e^\gamma\rho^n.
\end{align*}
Here we used $W_\chi-U-s=v-U+\varepsilon_{s,k}T^r$,
$\tau_k(t)\geq t$, and $v-U\geq v_\delta-U>m/4$.
	
	Take a smooth local potential $\theta$ of $\chi$ and write
	$e^\gamma\rho^n=g\beta^n$ on $B_R$. The function $g$ is nonnegative
	and continuous on $\overline{B_R}$. The functions $w+\theta$ and
	$U+\theta$ meet the hypotheses of
	Lemma~\ref{lem:strict-factor-comparison}. On the other hand,
	$$
	(w-U)(x_\delta)=m_\delta+\kappa R^2,
	\qquad
	\max_{\partial B_R}(w-U)\leq m_\delta,
	$$
	contradicting that lemma. Hence $v\leq U$.
\end{proof}

In particular, on $X_s\cap\Omega_\chi$ we have
$h-s\leq\varepsilon_{s,k} T^r$. Raising this inequality to the power
$1/r$, and using $W_\chi\leq0$, gives
\begin{equation}
	\label{eq:pointwise-level-comparison}
	r A_{s,k}^{-1/n}(h-s)^{(n+1)/n}
	\leq-\phi_{s,k}+rA_{s,k}.
\end{equation}
This holds almost everywhere on $X_s$, since
$X\setminus\Omega_\chi$ is a proper analytic set.

\subsection{Energy bounds and De Giorgi iteration}
\label{subsec:entropy-energy-iteration}

For the energy bound, we adapt the entropy tail estimate and
absorption argument of Guo--Phong~\cite{GuoPhongEntropyEnergy}.
Both
cases~\ref{enu:thm-exp} and~\ref{enu:thm-poly} of
Theorem~\ref{thm:determinant-relative-Linfty} reduce to a weighted
moment bound with exponent $\theta>n$, while the superlevel estimate
of Lemma~\ref{lem:fasterdecayphi} below only uses the order-one
entropy $\operatorname{Ent}_{\mathcal D,1}$. Throughout the rest of the
section we use
\[
\text{case~\ref{enu:thm-exp}: }\theta=q,\qquad
\text{case~\ref{enu:thm-poly}: }\theta=q/p'.
\]
\begin{lem}
	\label{lem:generalizedyoungs}
	For $t\geq0$, $b\in\mathbb R$, and $\sigma>0$,
\begin{equation}
	\label{eq:entropy-young}
	t^\sigma e^b\leq e^b(1+b_+)^\sigma+C_\sigma e^{2t}.
\end{equation}
\end{lem}
\begin{proof}
	If $t\leq1+b_+$, the first term suffices. Otherwise $b_+<t$,
	and $t^\sigma e^b\leq t^\sigma e^t\leq C_\sigma e^{2t}$.
\end{proof}

In the exponential case, apply
\eqref{eq:entropy-young} with $\sigma=q$, $t=-\alpha\varphi/2$, and
$b=\log f$ on $\{f>0\}$. It follows that
\begin{equation}
	\label{eq:entropy-weighted-moment}
	\sup_{\varphi\in\operatorname{PSH}_0(X,\chi)}
	\int_X(-\varphi)^\theta d\mu \leq M_\theta,
\end{equation}
with $M_\theta=C(\alpha,q)M$; also $\mu(X)\leq M$. The order-one
entropy is controlled, since $q>n\geq1$ and
$\operatorname{Ent}_{\mathcal D,1}\leq\operatorname{Ent}_{\mathcal D,q}$.

In the power case, H\"older's inequality gives the same bound with
\[
M_\theta=M,\qquad
\mu(X)\leq M^{1/p}\rho^n(X)^{1/p'}.
\]
This case also controls the order-one entropy. Indeed, the elementary
estimate
\[
e^b(1+b_+)\leq C_p(e^b+e^{pb}),\qquad b\in\mathbb R,
\]
applied with $b=\log f$, gives
\[
\operatorname{Ent}_{\mathcal D,1}(\chi,\gamma,\rho)
\leq C_p\bigl(\mu(X)+\int_Xf^p\rho^n\bigr)
\leq C(p,q,n,M,\int_X\rho^n).
\]

\begin{lem}
	\label{lem:fasterdecayphi}
	With the notation of \eqref{eq:singular-density-entropy}, suppose
	that
	\[
	\operatorname{Ent}_{\mathcal D,1}(\chi,\gamma,\rho)\leq H,
	\qquad
	\int_Xh\,\rho^n\leq L.
	\]
	Then
	\begin{equation}
		\label{eq:entropy-tail}
		\mu\{h>s\}\leq\frac{C(H+L)}{\log s},
		\qquad s>1.
	\end{equation}
\end{lem}
\begin{proof}
	On $X_s=\{h>s\}$ with $s>1$ we have $h>s>1$.
	Applying \eqref{eq:entropy-young} with $\sigma=1$,
	$t=\frac12\log h$, and $b=\log f$ on $X_s\cap\{f>0\}$, multiplying by $\rho^n$, and integrating, we obtain
	$$
	\begin{aligned}
	\tfrac12(\log s)\mu(X_s)
&\leq\int_{X_s}\frac{\log h}{2}\,d\mu\\
&\leq\int_Xf(1+\log_+f)\rho^n
+C\int_{X_s}h\,\rho^n\\
&=\operatorname{Ent}_{\mathcal D,1}(\chi,\gamma,\rho)
+C\int_{X_s}h\,\rho^n
\leq H+CL.
	\end{aligned}
	$$
Here we used
$\int_Xf(1+\log_+f)\rho^n
=\operatorname{Ent}_{\mathcal D,1}(\chi,\gamma,\rho)$,
which follows from $f=e^\gamma/V_\chi$ and
$\log_+f=(\gamma-\log V_\chi)_+$.
	This completes the proof.
\end{proof}

We formulate the next lemma in terms of the precise hypotheses used
in its proof, so that it also applies to
Theorem~\ref{thm:infinite-density-relative-Linfty} below.

\begin{lem}
	\label{lem:energygrowth-unified}
	With the notation of \eqref{eq:singular-density-entropy},
	\eqref{eq:level-set-quantities}, and
	\eqref{eq:auxiliary-normalization}, suppose
	$\mu(X)\leq M_0$, that \eqref{eq:entropy-weighted-moment}
	holds for some $\theta>n$, and that for each $s,k$ there is
	$\phi_{s,k}\in\operatorname{PSH}_0(X,\chi)$ satisfying
	\eqref{eq:pointwise-level-comparison} $\mu$-almost everywhere on $X_s$.
	Then:
	\begin{enumerate}[label=\textnormal{(\roman*)}]
		\item\label{enu:unified-moments}
		If the entropy and $L^1$ bounds of Lemma~\ref{lem:fasterdecayphi}
		hold, then $A_0\leq E_0$, where $E_0$ depends only on
		$(n,\theta,M_\theta,M_0,H,L)$.
		\item\label{enu:unified-growth}
		If $A_0\leq E_0$ is given, then
		\begin{equation}
			\label{eq:energy-and-level-growth}
			A_s\leq B_0\Phi(s)^{1+\delta},\qquad
			\delta=\frac1n-\frac1\theta>0,
		\end{equation}
		where $B_0$ depends only on $(n,\theta,M_\theta,M_0,E_0)$.
	\end{enumerate}
	In particular, these constants do not depend on the initial bound for $h$.
\end{lem}

\begin{proof}
	Put $r=n/(n+1)$. Raise \eqref{eq:pointwise-level-comparison} to
	$\theta$, integrate over $X_s$, and use the weighted moment bound:
	\[
	\int_{X_s}(h-s)^{\theta/r}d\mu
	\leq C(n,\theta)A_{s,k}^{\theta/n}
	\bigl(M_\theta+A_{s,k}^\theta\Phi(s)\bigr).
	\]
	Since $A_{s,k}\to A_s$, this gives
	\begin{equation}
		\label{eq:level-high-moment}
		\int_{X_s}(h-s)^{\theta/r}d\mu
		\leq C(n,\theta)A_s^{\theta/n}
		\bigl(M_\theta+A_s^\theta\Phi(s)\bigr).
	\end{equation}
	If $A_s=0$, the desired growth estimate is immediate. Otherwise
	H\"older gives
	\[
	A_s\leq\left(\int_{X_s}(h-s)^{\theta/r}d\mu\right)^{r/\theta}
	\Phi(s)^{1-r/\theta}.
	\]
	Dividing the resulting inequality by $A_s^{1/(n+1)}$ yields
	\begin{equation}
		\label{eq:energy-absorption}
		A_s^r\leq C_1\bigl(\Phi(s)^{1-r/\theta}+A_s^r\Phi(s)\bigr),
		\qquad C_1=C(n,\theta,M_\theta).
	\end{equation}
	For part~\ref{enu:unified-moments}, Lemma~\ref{lem:fasterdecayphi}
	gives a controlled $s_0>1$ with
	$C_1\Phi(s_0)\leq1/2$. Absorption bounds $A_{s_0}$, hence
	\[
	A_0\leq A_{s_0}+s_0\mu(X)\leq E_0.
	\]
	For part~\ref{enu:unified-growth}, use $A_s\leq E_0$ and
	$\Phi(s)\leq M_0$ in \eqref{eq:energy-absorption}. Its right side
	is bounded by $C\Phi(s)^{1-r/\theta}$, and
	$(1-r/\theta)/r=1+1/n-1/\theta$ gives the assertion.
\end{proof}

\begin{thm} \cite[Lemma~2]{MR4593734}
	\label{thm:Degiorgi}
	Let $\Phi:[0,\infty)\to[0,\infty)$ be decreasing and suppose
	\[
	t\Phi(s+t)\leq B\Phi(s)^{1+\delta}
	\quad(s\geq0,\ 0<t\leq1)
	\]
	for $B,\delta>0$. If $\Phi(s_0)^\delta\leq(2B)^{-1}$, then
	$\Phi(s)=0$ for $s\geq s_0+(1-2^{-\delta})^{-1}$.

\end{thm}
\begin{proof}
	Set $s_{j+1}=s_j+2B\Phi(s_j)^\delta$ until $\Phi(s_j)=0$.
	Otherwise induction gives $\Phi(s_{j+1})\leq\Phi(s_j)/2$ and
	$s_{j+1}-s_j\leq2^{-j\delta}$. Thus
	$s_j\uparrow s_\infty\leq s_0+(1-2^{-\delta})^{-1}$, and
	monotonicity gives $\Phi(s_\infty)\leq\Phi(s_j)\to0$.
\end{proof}

\begin{proof}[Proof of Theorem~\ref{thm:determinant-relative-Linfty}]
	The hypotheses of Lemma~\ref{lem:energygrowth-unified} are
	satisfied with the exponent $\theta$ fixed above.
	We obtain $A_0\leq E_0$ and \eqref{eq:energy-and-level-growth}, with
	exactly the dependence stated in the theorem. Since
	\[
	t\Phi(s+t)\leq A_s,\qquad \Phi(t)\leq E_0/t,
	\]
	choose $s_0\geq1+E_0(2B_0)^{1/\delta}$. Theorem~\ref{thm:Degiorgi}
	gives $\Phi(S)=0$ for a controlled $S$.
	
	Because $f>0$ almost everywhere on $X\setminus Z$, this implies
	$W_\chi\leq U+S$ almost everywhere there. Since $W_\chi$ is
	$\chi$-psh and $U+S$ is continuous, the almost-everywhere
	inequality holds everywhere on $X\setminus Z$ (add a local
	smooth potential of $\chi$ and use the submean inequality on
	shrinking balls). On $Z$ the inequality is automatic.
\end{proof}

\begin{thm}
	\label{thm:infinite-density-relative-Linfty}
	Theorem~\ref{thm:determinant-relative-Linfty} remains valid if
	$\gamma\in C(X\setminus Z,[-\infty,+\infty])$ in the extended sense,
	provided
	\[
	S=\{x\in X\setminus Z:\gamma(x)=+\infty\}
	\]
	is locally pluripolar. All other qualitative hypotheses and the
	bounds in case~\ref{enu:thm-exp} or case~\ref{enu:thm-poly} remain
	unchanged, and the constant has the same dependence.
\end{thm}

\begin{proof}
	The auxiliary equation, its solution with minimal singularities,
	and Lemma~\ref{lem:comparison-function-psh} remain valid: their
	measure data are unchanged by values on $S$. We only need to prove
	$v\leq U$, and hence \eqref{eq:pointwise-level-comparison} follows.
	
	Suppose $m=\sup_X(v-U)>0$. As in the proof of
	Lemma~\ref{lem:comparison-function-below-U}, choose $\zeta_0\leq v$, a small
	$\delta>0$, and $K_{s,k}$ with $T^r\leq K_{s,k}$, so that
	$v_\delta=(1-\delta)v+\delta\zeta_0$ has a maximum difference
	$m_\delta>m/2$ at $x_\delta\in\Omega_\chi\setminus Z$, and
	\begin{equation}
		\label{eq:extension-strict-factor}
		c_*=(1-\delta)^n\left(1+\frac{m}{4\varepsilon_{s,k} K_{s,k}}\right)>1.
	\end{equation}
	Take a coordinate ball $B_R=B_R(x_\delta)\Subset\Omega_\chi\setminus Z$
	with $v_\delta-U>m/4$ on its closure. If $x_\delta\notin S$, choose
	$B_R$ disjoint from $S$ and use the proof of
	Lemma~\ref{lem:comparison-function-below-U} directly. Otherwise local
	pluripolarity, after shrinking the ball, gives a psh function
	$g\leq0$ on a neighborhood of $\overline B_R$, not identically
	$-\infty$, with $S\cap\overline B_R\subset\{g=-\infty\}$.
	
	Set $g_N=\max\{g,-N\}$ and choose $\kappa_1>0$ with
	$\kappa_1\beta\leq\delta\delta_0\rho/2$. Choose a nearby point
	$y$ where $g(y)>-\infty$ and
	\[
		v_\delta(y)-U(y)+\kappa_1(R^2-|y-x_\delta|^2)
		>m_\delta+\kappa_1R^2/2,
	\]
	then choose $t>0$ so small that $t g(y)>-\kappa_1R^2/4$. For every $N$,
	the bounded local $\chi$-psh function
	\[
	w_N=v_\delta+\kappa_1(R^2-|z-x_\delta|^2)+t g_N
	\]
	satisfies
	\[
	(w_N-U)(y)>m_\delta+\kappa_1R^2/4,
	\qquad \limsup_{z\to\partial B_R}(w_N-U)\leq m_\delta.
	\]
	On $S\cap B_R$ it satisfies
	$w_N-U\leq m_\delta+\kappa_1R^2-tN$.
	Choose $N$ with $tN>\kappa_1R^2$. The upper semicontinuous
	function $w_N-U$ then attains its maximum at an interior point
	$x_*\notin S$. Extended continuity of $\gamma$ makes $S$ relatively
	closed, so choose $B_r(x_*)\Subset B_R\setminus S$.
	
	Choose $\kappa_2>0$ with $\kappa_2\beta\leq\delta\delta_0\rho/2$
	and put $\widetilde w=w_N-\kappa_2|z-x_*|^2$. Then
	by \eqref{eq:vdelta-positive-current}
	\[
	\chi+\ddc\widetilde w
	\geq(1-\delta)r\varepsilon_{s,k} T^{r-1}(\chi+\ddc\phi_{s,k}).
	\]
By the same calculation as in
Lemma~\ref{lem:comparison-function-below-U},
	\[
	(\chi+\ddc\widetilde w)^n_{\mathrm{BT}}\geq c_*e^\gamma\rho^n
	\quad\hbox{on }B_r(x_*).
	\]
	The density is finite, nonnegative, and continuous on this closed
	ball. After adding a local potential of $\chi$,
	Lemma~\ref{lem:strict-factor-comparison} applies to
	$\widetilde w$ and $U$. But their difference at $x_*$ exceeds its
	boundary limsup by at least $\kappa_2r^2$, a contradiction.
	
	Thus \eqref{eq:pointwise-level-comparison} holds. The weighted
	moment bounds, Lemma~\ref{lem:energygrowth-unified}, and the
	De Giorgi argument now give the asserted estimate with the same
	constants. The choices of $g,N,t,\delta$, and the balls enter only
	the qualitative comparison proof.
\end{proof}

\subsection{Consequences for \texorpdfstring{$\mathcal D$}{D}-subsolutions}
\label{subsec:singular-D-relative-estimate}

Proposition~\ref{prop:determinant-reduction} transfers the scalar
estimate to differences of potentials. For a smooth admissible
solution, the reference need only be a viscosity
$\mathcal D$-subsolution. Local regularizability is required when
the supersolution is nonsmooth.

\begin{thm}
	\label{thm:singular-D-relative-Linfty}
	Let $(F,\mathcal A)$ be an admissible pair that satisfies QSH, and let $u\in C^2(X)$
	be a normalized admissible solution of \eqref{eq:main-equation}.
	Let $[\chi]$ be big and let $Z\subset X$ be a closed proper set.
	Let $\gamma\in C(X\setminus Z,\mathbb R)$ and suppose that
	$\underline u\in C(X\setminus Z)$ is a singular viscosity
	$\mathcal D$-subsolution with these data and $\underline u\leq0$.
	Assume that the zero extension
	of $e^\gamma$ belongs to $L^{p_0}(X,\rho^n)$ for some $p_0>1$.
	If $q>n$ and
	$$
	\operatorname{Ent}_{\mathcal D,q}(\chi,\gamma,\rho)\leq H,
	$$
	then
	\begin{equation}
		\label{eq:singular-D-relative-bound}
		u\geq\underline u+W_\chi-C
		\quad\hbox{on }X\setminus Z,
	\end{equation}
	where
	$$
	C=C\bigl(X,\rho,\chi,q,H,
	b_{\mathcal A}+\|\operatorname{tr}_\rho\omega\|_{L^\infty(X)}\bigr).
	$$
	In particular, $u$ is bounded on each compact subset of
	$\Omega_\chi\setminus Z$ by the displayed data and a lower bound
	for $\underline u+W_\chi$ there.
\end{thm}

\begin{proof}
	The last assertion of Proposition~\ref{prop:determinant-reduction}
	shows that $U=u-\underline u$ satisfies the viscosity determinant
	bound on $X\setminus Z$. Since $u$ is continuous on $X$ and
	$\underline u\to-\infty$ along $Z$, its extension by $+\infty$
	is continuous in the extended sense. Moreover,
	$$
	(W_\chi-U)_+=(W_\chi+\underline u-u)_+\leq-u,
	$$
	because $W_\chi\leq0$, $\underline u\leq0$, and $\sup_Xu=0$.
	Lemma~\ref{lem:admissible-L1}
	provides the $L^1$ bound
	required by Theorem~\ref{thm:determinant-relative-Linfty}.
	Lemma~\ref{lem:uniform-alpha-integrability} supplies the exponential
	moment bound for the fixed smooth form $\chi$. Case~\ref{enu:thm-exp}
	therefore proves \eqref{eq:singular-D-relative-bound}.
\end{proof}

Local regularizability implies the viscosity $\mathcal D$-subsolution
property by Proposition~\ref{prop:D-consistency-stability},
but is not required in Theorem~\ref{thm:singular-D-relative-Linfty}.
When a subsolution is constructed by regularization, neither the
sequence nor the size of its derivatives enters the estimate.
The singular set need not be analytic; analyticity is useful in
applications for constructing the subsolution.

\begin{cor}
	\label{cor:singular-D-uniform-family}
	Let $(F_j,\mathcal A_j,\omega_j,\psi_j,u_j,\underline u_j,\gamma_j)$
	satisfy the hypotheses of
	Theorem~\ref{thm:singular-D-relative-Linfty} with the same
	$(X,\rho,\chi,Z,q)$. The qualitative exponents $p_{0,j}>1$ may vary.
	If
	$$
	\sup_j\operatorname{Ent}_{\mathcal D,q}(\chi,\gamma_j,\rho)\leq H,
	\qquad
	\sup_j\left(b_{\mathcal A_j}
	+\|\operatorname{tr}_\rho\omega_j\|_{L^\infty(X)}\right)\leq B,
	$$
	then
	$$
	u_j\geq\underline u_j+W_\chi-C(X,\rho,\chi,q,H,B)
	\quad\hbox{on }X\setminus Z.
	$$
\end{cor}

\begin{proof}
	The constants in Theorem~\ref{thm:singular-D-relative-Linfty} are
	uniform under these assumptions.
\end{proof}

For a continuous viscosity supersolution $\overline u$ of
\eqref{eq:main-equation}, assume in addition that $\underline u$
is locally regularizable. Proposition~\ref{prop:determinant-reduction}
and Theorem~\ref{thm:determinant-relative-Linfty} give
$\overline u\geq\underline u+W_\chi-C$ whenever
$U=\overline u-\underline u$ satisfies the
continuity, boundary divergence, and relative $L^1$ hypotheses of
that theorem. These hypotheses are not inferred merely from the
poles of $\underline u$.

The power case of Theorem~\ref{thm:determinant-relative-Linfty}
also applies to these differences whenever its density and potential
moment bounds hold.

\begin{proof}[Proof of Theorem~\ref{thm:intro-relative-Linfty}]
	Apply Theorem~\ref{thm:singular-D-relative-Linfty} with
	$Z=\varnothing$ and $\underline u=0$.
	By Proposition~\ref{prop:D-consistency-stability}, the zero
	function is a viscosity $\mathcal D$-subsolution with the stated
	data. Continuity of $\gamma$ ensures the required integrability,
	so the estimate follows with the stated constant dependence.
	
	If $\chi\geq0$, $W_\chi=0$. Together with
	$\sup_Xu=0$, this yields the claimed $L^\infty$ bound.
\end{proof}

The following result is one of the main consequences.

\begin{cor}
	\label{cor:det-elliptic-Linfty}
	Suppose that $(F,\mathcal A)$ satisfies QSH and
	\eqref{eq:det-elliptic}
	with constants $\eta,\beta>0$. Let $u\in C^2(X)$ solve
	\eqref{eq:main-equation}, and suppose that $[\chi]$ is big and
	$\omega-\chi\in\mathcal A$. Define
	\begin{equation}
		\label{eq:det-elliptic-density}
		d=\left[\eta^{-1}\bigl(\psi-F_x(\omega-\chi)\bigr)_+\right]^{1/\beta},
		\qquad f=\frac d{V_\chi}.
	\end{equation}
	If $\int_X f(1+\log_+f)^q\rho^n\leq H$ for some $q>n$,
	then the conclusions of
	Theorem~\ref{thm:intro-relative-Linfty} hold.

For a G{\aa}rding polynomial operator $F=G$ of degree $m\geq1$,
the same conclusion holds with $\beta=m/n$ and
$\omega-\chi\in\mathcal C_G^{(1)}$, provided
$G_x^{\mathrm{top}}(P)\geq\eta\det_\rho(P)^{m/n}$ for $P\geq0$.
QSH is still required on the solution's admissible region.
\end{cor}

\begin{proof}
	By Proposition~\ref{prop:det-elliptic-D-subsolution}, or Corollary~\ref{cor:polynomial-D-subsolution} in the polynomial case, the function
	$d$ in \eqref{eq:det-elliptic-density} bounds $\det_\rho P$ for every
	$P$ satisfying \eqref{eq:D-subsolution-test} with $\underline u=0$.
	To allow zeros of $d$, set
	$$
	d_\varepsilon=d+\varepsilon V_\chi,
	\qquad \gamma_\varepsilon=\log d_\varepsilon,
	\qquad 0<\varepsilon\leq1.
	$$
	Then $(\omega,\chi,\gamma_\varepsilon)$ is a $\mathcal D$-subsolution
	with a positive continuous density. Its normalized density is
	$f+\varepsilon$, and
	$$
	\int_X(f+\varepsilon)(1+\log_+(f+\varepsilon))^q\rho^n
	\leq C_q\bigl(H+\rho^n(X)\bigr).
	$$
	Theorem~\ref{thm:intro-relative-Linfty}
	applies with constants independent of $\varepsilon$.
\end{proof}

\begin{rem}
\label{rem:det-elliptic-degenerate-eta}
In Corollary~\ref{cor:det-elliptic-Linfty}, the coefficient $\eta$
may be continuous and nonnegative, with proper locally pluripolar
zero set $S:=\{\eta=0\}$. Retain the reference condition
$\omega-\chi\in\mathcal A$ and impose \eqref{eq:det-elliptic}
on $X\setminus S$.
For a G{\aa}rding polynomial operator $F=G$ of degree $m$,
these conditions may be replaced by
\[
\omega-\chi\in\mathcal C_G^{(1)},\qquad
G_x^{\mathrm{top}}(P)\geq\eta(x)\det_\rho(P)^{m/n}
\quad(x\notin S,\ P\geq0),
\]
with $\beta=m/n$.

Proposition~\ref{prop:det-elliptic-D-subsolution} gives the
determinant test bound $\det_\rho P\leq d$ on $X\setminus S$, where
\[
d=\left[\eta^{-1}
\bigl(\psi-F_x(\omega-\chi)\bigr)_+\right]^{1/\beta}.
\]
Corollary~\ref{cor:polynomial-D-subsolution} gives the same bound
in the polynomial case.

Choose $\gamma\in C(X,\mathbb R\cup\{+\infty\})$, continuous
in the extended sense, with $\{\gamma=+\infty\}=S$ and
$e^\gamma\geq d$ on $X\setminus S$, and set $f=e^\gamma/V_\chi$.
The determinant bound is then automatic on $S$.
If $f\in L^{p_0}(X,\rho^n)$ for some $p_0>1$ and
\[
\int_X f(1+\log_+f)^q\rho^n\leq H,\qquad q>n,
\]
Theorem~\ref{thm:infinite-density-relative-Linfty} applies with
$Z=\varnothing$ and $U=u$.
The required $L^1$ bound follows from
$(W_\chi-u)_+\leq-u$ and Lemma~\ref{lem:admissible-L1};
the exponential integrability follows from
Lemma~\ref{lem:uniform-alpha-integrability}.
\end{rem}

The explicit constant in Theorem~\ref{thm:determinant-relative-Linfty}
gives uniform estimates for the degenerating semipositive
backgrounds used in the examples below.

\begin{cor}
	\label{cor:semipositive-uniform-Linfty}
	Let $\chi_j$ be smooth closed real $(1,1)$-forms satisfying
	$$
	0\leq\chi_j\leq K\rho,
	\qquad
	V_{\chi_j}=\int_X\chi_j^n\geq v_0>0.
	$$
	Let $\gamma_j\in C(X,\mathbb R)$ and $U_j\in C(X)$ with
	$\sup_XU_j=0$. Suppose that $U_j$ satisfies the viscosity
	determinant bound for $(\chi_j,\gamma_j)$ and that, for some $q>n$,
	$$
	\sup_j\operatorname{Ent}_{\mathcal D,q}(\chi_j,\gamma_j,\rho)\leq H,
	\qquad
	\sup_j\int_X(-U_j)\rho^n\leq L.
	$$
	Then
	$$
	\sup_j\|U_j\|_{L^\infty(X)}
	\leq C(X,\rho,K,v_0,q,H,L).
	$$
	No positive lower bound for $\chi_j$ as a form is required.
\end{cor}

\begin{proof}
	Each $[\chi_j]$ is big and $W_{\chi_j}=0$. By
	Lemma~\ref{lem:uniform-alpha-integrability}, the exponential
	integrability constants are uniform. The qualitative
	$L^{p_0}$ condition holds for each continuous density separately.
	The explicit constant in
	Theorem~\ref{thm:determinant-relative-Linfty} therefore gives the
	assertion.
\end{proof}

In particular, if $\chi_j$ converge smoothly to a semipositive big
form and the determinant densities are uniformly bounded, the
entropy bounds in Corollary~\ref{cor:semipositive-uniform-Linfty}
are automatic. The ABP estimate in Theorem~\ref{thm:ABP-explicit-delta} assumes
$\chi\geq\delta\rho$ and depends on this lower positivity constant.

\begin{rem}
	\label{rem:qiao-entropy}
	Under the structural hypotheses in \cite{MR4913058}, Qiao obtains
	$L^\infty$ estimates at the weaker scale
	$L^1(\log L)^n(\log\log L)^r$, $r>n$.
	We retain the simpler condition $L^1(\log L)^q$, $q>n$, here;
	extending the singular relative estimate to that scale would
	require a corresponding refinement of the entropy and iteration
	arguments.
\end{rem}

\section{\texorpdfstring{${\mathcal K}$}{K}-\gar{} polynomials and \gar{} elliptic operators}
\label{sec:K-Garding}

\gar{} polynomials are introduced in \cite{Fang-MaGard}. This family of polynomials is invariant under positive affine maps and is designed to capture ellipticity as an algebraic condition.

To apply G{\aa}rding theory to the PDEs studied in this paper, we
extend the notion of \gar{} polynomials to the following conic
setting.

Let \(W\) be a real finite-dimensional vector space. We use $A,B,\cdots$ to represent elements of $W$. For a polynomial \(F\in\mathbb{R}[W]\) and a vector \(P\in W\), we write
\[
\pdv_P F(A)
:=
\left.\frac{d}{dt}\right|_{t=0}F(A+tP)
\]
for the directional derivative of \(F\) in the direction \(P\).
\subsection{Conic ray test}
We introduce the conic ray test, which generalizes the positive ray
test defined in \cite{fang2026idealgaardingpolynomials} by replacing the positive orthant with a prescribed convex cone.
\begin{defn}[Conic ray test]
\label{def:K-CRT}
Let \(W\) be a finite-dimensional real vector space, and let
\(\mathcal K\subsetneq W\) be a \emph{closed} convex cone with nonempty
interior. A subset
\(\mathcal U\subseteq W\) is said to pass the
\emph{\(\mathcal K\)-conic ray test}, or \emph{\(\mathcal K\)-CRT},
if
\[
\mathcal U+\mathcal K\subseteq\mathcal U.
\]
\end{defn}
\begin{lem}
\label{lem:uniqueness-K-component}
Notations as above. Let \(F\in\mathbb{R}[W]\) be nonzero. Then \(\{F>0\}\) has at most
one connected component satisfying the \(\mathcal K\)-conic ray
test.
\end{lem}
\begin{proof}
Suppose that \(\mathcal C_1\) and \(\mathcal C_2\) are two such
components. Choose \(A_i\in\mathcal C_i\) and
\(P\in\operatorname{int}\mathcal K\). Then
\[
P+\frac{A_2-A_1}{T}\longrightarrow P
\qquad\text{as }T\longrightarrow+\infty,
\]
which implies that
\(
P+\frac{A_2-A_1}{T}\in\operatorname{int}\mathcal K
\)
for all sufficiently large \(T\). Since \(\mathcal K\) is a cone,
\(
TP+A_2-A_1\in\mathcal K.
\)
Hence
\[
A_2+TP
=
A_1+(TP+A_2-A_1)
\in\mathcal C_1.
\]
Since also \(TP\in\mathcal K\), we have
\(
A_2+TP\in\mathcal C_2.
\)
Thus
\(
\mathcal C_1\cap\mathcal C_2\neq\varnothing,
\)
and therefore
\(
\mathcal C_1=\mathcal C_2.
\)
\end{proof}
\begin{lem}
\label{lem:K-contained-homogeneous-component}
Let \(\mathcal K\subseteq W\) be a closed convex cone with nonempty
interior, and let \(F\in\mathbb{R}[W]\) be a nonzero homogeneous
polynomial. Suppose that \(\mathcal C\) is a nonempty connected
component of \(\{F>0\}\) satisfying \(\mathcal K\)-CRT. Then
\[
\operatorname{int}\mathcal K\subseteq\mathcal C
\qquad\text{and}\qquad
\mathcal K\subseteq\overline{\mathcal C}.
\]
\end{lem}
\begin{proof}
Fix \(P\in\operatorname{int}\mathcal K\) and choose
\(A\in\mathcal C\). Since \(\operatorname{int}\mathcal K\) is open,
for all sufficiently small \(\varepsilon>0\),
\begin{equation}\label{add2}
P-\varepsilon A\in\operatorname{int}\mathcal K
\subseteq\mathcal K.\end{equation}
Since \(F\) is homogeneous, \(\mathcal C\) is invariant under
positive scaling, and hence
\begin{equation}\label{add1}
\varepsilon A\in\mathcal C.
\end{equation}
By \eqref{add2}, \eqref{add1} and \(\mathcal K\)-CRT,
\(
P
=
\varepsilon A+(P-\varepsilon A)
\in\mathcal C.
\)
Thus
\(
\operatorname{int}\mathcal K\subseteq\mathcal C.
\)
Since \(\mathcal K\) is closed and convex with nonempty interior,
\(
\mathcal K\subseteq\overline{\mathcal C}.
\)
\end{proof}
\begin{defn}[\(\mathcal K\)-G{\aa}rding component]
\label{def:K-Garding-component}
Let \(F\in\mathbb{R}[W]\) be a nonzero polynomial. If \(\{F>0\}\)
has a connected component satisfying the \(\mathcal K\)-conic ray
test, we denote this unique component by
\(\mathcal C_F^{\mathcal K}\) and call it the
\emph{\(\mathcal K\)-G{\aa}rding component} of \(F\).
\end{defn}
For a positive constant polynomial \(F\), we adopt the convention
\(
\mathcal C_F^{\mathcal K}=W.
\)
\subsection{\(\mathcal K\)-G{\aa}rding polynomials}
We extend the recursive definition of G{\aa}rding polynomials
by replacing the positive coordinate directions with the directions
in \(\mathcal K\).
\begin{defn}[\(\mathcal K\)-G{\aa}rding polynomial]
\label{def:K-Garding}
A nonzero polynomial \(F\in\mathbb{R}[W]\) is called
\emph{\(\mathcal K\)-G{\aa}rding} if its
\(\mathcal K\)-G{\aa}rding component
\(\mathcal C_F^{\mathcal K}\) exists and, whenever
\(\deg F\geq 1\), for every \(P\in\mathcal K\) such that
\(\pdv_P F\not\equiv 0\), the polynomial \(\pdv_P F\) is
\(\mathcal K\)-G{\aa}rding and
\begin{equation}\label{eq:inclusion}
\mathcal C_F^{\mathcal K}
\subseteq
\mathcal C_{\pdv_P F}^{\mathcal K}.
\end{equation}
We denote the class of \(\mathcal K\)-G{\aa}rding polynomials on
\(W\) by
\(
\mathrm{G}^{\mathcal K}[W].
\)
\end{defn}
\eqref{eq:inclusion} in Definition~\ref{def:K-Garding} is the
analogue of the nesting condition
\(
\mathcal C_F\subseteq\mathcal C_{\partial_i F}
\)
for  G{\aa}rding polynomials.

Let \(F\) be a \(\mathcal K\)-G{\aa}rding polynomial of degree \(d\).
For \(0\leq j\leq d\), define
\[
\bigl(\mathcal C_F^{\mathcal K}\bigr)^{(j)}
:=
\bigcap_{\substack{
P_1,\ldots,P_j\in\mathcal K\\
\pdv_{P_1}\cdots\pdv_{P_j}F\not\equiv 0
}}
\mathcal C_{\pdv_{P_1}\cdots\pdv_{P_j}F}^{\mathcal K},
\]
with
\[
\bigl(\mathcal C_F^{\mathcal K}\bigr)^{(0)}
:=
\mathcal C_F^{\mathcal K}.
\]
These components satisfy
\[
\mathcal C_F^{\mathcal K}
=
\bigl(\mathcal C_F^{\mathcal K}\bigr)^{(0)}
\subseteq
\bigl(\mathcal C_F^{\mathcal K}\bigr)^{(1)}
\subseteq
\cdots
\subseteq
\bigl(\mathcal C_F^{\mathcal K}\bigr)^{(d)}
=
W.
\]

\begin{rem}
\label{rem:derived-admissible-region}
The derived sets are upward closed and contain
$\mathcal C_F^{\mathcal K}$. For analytic applications, we choose
\[
\mathcal A_F:=\operatorname{int}\bigl((\mathcal C_F^{\mathcal K})^{(1)}\bigr),
\qquad
\mathcal C_F^{\mathcal K}\subseteq\mathcal A_F
\subseteq(\mathcal C_F^{\mathcal K})^{(1)}.
\]
The Taylor increment estimate holds on the full first-derived set,
which may therefore contain the reference forms used to construct
$\mathcal D$-subsolutions. QSH is imposed separately. 
\end{rem}

\subsection{Scalarization and top homogeneous parts}

We use the recursive characterization of ordinary G{\aa}rding
polynomials, their closure under positive directional derivatives,
and the nesting of their distinguished components from
\cite{Fang-MaGard}. Thus the ordinary theory is the case
$\mathcal K=\overline{\Gamma_m^+}$ of the conic definition.

\begin{thm}
\label{thm:affine-scalarization-K-Garding}
Let $F\in\GS^{\mathcal K}[W]$, let $P_1,\ldots,P_m\in\mathcal K$,
and set
\begin{equation}\label{add3}
L(\x)=A+\sum_i x_iP_i,\qquad f=F\circ L.
\end{equation}
If    $L(\mathbb R^m)\cap\mathcal C_F^{\mathcal K}\neq\varnothing$,   then $f$
is G{\aa}rding. This holds, in particular, if
$A\in\mathcal C_F^{\mathcal K}$ or
$\sum_iP_i\in\operatorname{int}\mathcal K$.
If $A\in\mathcal C_F^{\mathcal K}$, then $f\in\GS_+[\x]$.
\end{thm}
\begin{proof}
Induct on $\deg F$, the constant case being immediate.
The nonempty set $C=L^{-1}(\mathcal C_F^{\mathcal K})$ is upward
closed, hence connected: any two points join a common coordinatewise
upper bound. It is a component of $\{f>0\}$, since every connected
subset of $\{f>0\}$ meeting $C$ maps into $\mathcal C_F^{\mathcal K}$.
For each nonzero $\partial_{P_i}F$, the pullback
$\partial_i f=(\partial_{P_i}F)\circ L$ is positive on $C$.
Induction and the same preimage description give
$C\subseteq\mathcal C_{\partial_i f}$ by component nesting. The recursive
criterion proves that $f$ is G{\aa}rding.

If $Q=\sum_iP_i\in\operatorname{int}\mathcal K$, choose
$B\in\mathcal C_F^{\mathcal K}$. For large $t$,
$A+tQ-B\in\mathcal K$, so $L(t\one)\in\mathcal C_F^{\mathcal K}$.
If $A\in\mathcal C_F^{\mathcal K}$, recursive nesting gives
\[
[\x^\alpha]f
=\frac{\partial_{P_1}^{\alpha_1}\cdots
\partial_{P_m}^{\alpha_m}F(A)}{\alpha!}\geq0.
\]
\end{proof}

\begin{cor}
\label{cor:homogeneous-K-scalarization}
Let $H$ be homogeneous and $\mathcal K$-G{\aa}rding of degree
$d\geq1$. If $P_i\in\mathcal K$ and
$\sum_iP_i\in\operatorname{int}\mathcal K$, then
\[
h(\x)=H\left(\sum_i x_iP_i\right)
\]
is homogeneous G{\aa}rding of degree $d$, with nonnegative coefficients.
\end{cor}
\begin{proof}
By Lemma~\ref{lem:K-contained-homogeneous-component},
$\sum_iP_i\in\mathcal C_H^{\mathcal K}$. Thus
Theorem~\ref{thm:affine-scalarization-K-Garding} applies at $A=0$.
Since $h(\one)>0$, its degree is $d$; homogeneous G{\aa}rding
polynomials have nonnegative coefficients \cite{Fang-MaGard}.
\end{proof}

\begin{lem}
\label{lem:homogeneous-conic-reconstruction}
Let $H\in\mathbb R[W]$ be homogeneous and positive on
$\operatorname{int}\mathcal K$. Suppose
$H(\sum_i x_iP_i)$ is G{\aa}rding for every finite family
$P_i\in\mathcal K$ with $\sum_iP_i\in\operatorname{int}\mathcal K$.
Then $H$ is $\mathcal K$-G{\aa}rding.
\end{lem}
\begin{proof}
Induct on $\deg H$, starting with positive constants. Let $C$ be
the component of $\{H>0\}$ containing $\operatorname{int}\mathcal K$.
For $P\in\mathcal K$, choose $B_i\in\operatorname{int}\mathcal K$
spanning $W$ and set
\[
L(s,\x)=sP+\sum_i x_iB_i,\qquad q=H\circ L.
\]
Surjectivity of $L$ makes $L^{-1}(C)$ connected; it contains the
positive orthant and therefore equals $\mathcal C_q$.
The ray test and component nesting for $q$ give
$C+P\subset C$ and, whenever $J:=\partial_PH\not\equiv0$,
$J>0$ on $C$.

For such a nonzero $J$ and any $Q_i\in\mathcal K$ with
$Q=\sum_iQ_i$ interior, set
$r(s,\x)=H(sP+\sum_i x_iQ_i)$.
The point $(0,\one)$ lies in $\mathcal C_r$ because it is a limit
of positive points and $H(Q)>0$. Since $J(Q)>0$, component nesting
places it in $\mathcal C_{\partial_s r}$.
Apply Theorem~\ref{thm:affine-scalarization-K-Garding} to
$\partial_s r$ on the section $s=0$: the scalarization
$J(\sum_i x_iQ_i)$ is G{\aa}rding.
Induction makes $J$ $\mathcal K$-G{\aa}rding. Since $J>0$ on $C$
and both components contain $\operatorname{int}\mathcal K$,
$C\subseteq\mathcal C_J^{\mathcal K}$.
This proves the recursive criterion.
\end{proof}

\begin{prop}
\label{prop:conic-top-part}
If $F$ is $\mathcal K$-G{\aa}rding of degree $d\geq1$, then
$H=F^{\mathrm{top}}$ is $\mathcal K$-G{\aa}rding and positive on
$\operatorname{int}\mathcal K$.
\end{prop}
\begin{proof}
Fix $A\in\mathcal C_F^{\mathcal K}$ and $Q\in\operatorname{int}\mathcal K$.
Choose interior vectors $B_i$ spanning $W$ with $\sum_iB_i=Q$,
using small opposite perturbations of $Q/(2\dim W)$.
By Theorem~\ref{thm:affine-scalarization-K-Garding},
$f(\x)=F(A+\sum_i x_iB_i)$ has nonnegative coefficients;
surjectivity gives $\deg f=d$. Hence $H(Q)=f^{\mathrm{top}}(\one)>0$.

For any $P_i\in\mathcal K$ with interior sum,
$F(A+\sum_i x_iP_i)$ is G{\aa}rding of degree $d$.
Its top part $H(\sum_i x_iP_i)$ is G{\aa}rding by top-degree
truncation, polarization, and diagonal specialization
\cite{Fang-MaGard}. Lemma~\ref{lem:homogeneous-conic-reconstruction}
therefore applies.
\end{proof}

\begin{lem}
\label{lem:positive-linear-pullback}
Let $p\in\GS[\mathbb R^m]$ and let $L:W\to\mathbb R^m$ be
surjective with $L(\mathcal K)\subseteq\mathbb R_+^m$.
Then $p\circ L$ is $\mathcal K$-G{\aa}rding, with component
$L^{-1}(\mathcal C_p)$.
\end{lem}
\begin{proof}
Surjectivity makes $L^{-1}(\mathcal C_p)$ a connected component
of $\{p\circ L>0\}$; the cone inclusion makes it upward closed.
The identity
$\partial_P(p\circ L)=(\partial_{L(P)}p)\circ L$,
ordinary directional-derivative closure, and component nesting
prove the assertion by induction on degree.
\end{proof}

\subsection{G{\aa}rding elliptic operator}
We apply the $\mathcal{K}$-\gar{} construction in the following setting.

Let \(V\) be a complex vector space of dimension \(n\), and let \(W=Q(V)\) be the real vector space of Hermitian forms on \(V\). For convenience, we fix a nondegenerate positive Hermitian form \(\rho\) on \(V\). Therefore, $\rho\in Q(V)^{++}$.

\begin{defn}[G{\aa}rding elliptic operator]\label{def:Garding-elliptic-operator}A polynomial\[ G:Q(V)\longrightarrow\mathbb R\]is called a \emph{G{\aa}rding elliptic operator} if it is \(Q(V)^+\)-G{\aa}rding.\end{defn}
For a G{\aa}rding elliptic operator, we abbreviate
\[
\mathcal C_G := \mathcal C_G^{Q(V)^+}
\qquad\text{and}\qquad
\mathcal C_G^{(j)} := \bigl(\mathcal C_G^{Q(V)^+}\bigr)^{(j)}.
\]

We also set $\mathcal A_G:=\operatorname{int}\mathcal C_G^{(1)}$.

\begin{prop}
\label{prop:positive-branch-admissible}
Let $G$ be a G{\aa}rding elliptic operator of degree $d\geq1$.
Then $(G,\mathcal A_G)$ satisfies \textnormal{(A1)--(A3)}. For
$A\in\mathcal C_G^{(1)}$ and $P\geq0$,
\begin{equation}\label{eq:conic-taylor-increment}
G(A+P)-G(A)\geq G^{\mathrm{top}}(P).
\end{equation}
Moreover, $\mathcal C_G^{(1)}\cap\{G>0\}=\mathcal C_G$.
\end{prop}
\begin{proof}
Recursive nesting makes all positive-order directional Taylor
coefficients nonnegative on $\mathcal C_G^{(1)}$. Hence
\[
G(A+P)-G(A)=\sum_{k=1}^d\frac{\partial_P^kG(A)}{k!}
\geq\frac{\partial_P^dG(A)}{d!}=G^{\mathrm{top}}(P).
\]
The top part is positive for $P>0$ by
Proposition~\ref{prop:conic-top-part}. Since
$\mathcal C_G\subseteq\mathcal A_G$ and the interior of an upward
closed set is upward closed, \textnormal{(A1)--(A2)} follow.

Suppose $A_0\in\partial\mathcal A_G$ and $G(A_0)>0$.
The Taylor coefficients at $A_0$ remain nonnegative by continuity,
so $G(A_0+tQ)\geq G(A_0)>0$ for $Q>0$ and $t\geq0$.
Fix $B\in\mathcal C_G$. For large $t$, $A_0+tQ-B>0$, whence
$A_0+tQ\in\mathcal C_G$. The positive ray connects $A_0$ to this
component, giving $A_0\in\mathcal C_G\subseteq\mathcal A_G$, a
contradiction. Thus $G\leq0$ on $\partial\mathcal A_G$, proving
\textnormal{(A3)}. The same ray argument proves the last assertion.
\end{proof}

\begin{cor}
\label{cor:polynomial-D-subsolution}
Let $(G,\mathcal A)$ be an admissible pair of polynomial operators
of degree $d\geq1$, with each $G_x$ G{\aa}rding elliptic, and suppose
\[
G_x^{\mathrm{top}}(P)\geq\eta\det_\rho(P)^{d/n}
\qquad(P\geq0)
\]
for some $\eta>0$. If a smooth reference form
$B=\omega_{\underline u}-\chi$ lies pointwise in
$\mathcal C_{G_x}^{(1)}$, then $(\omega_{\underline u},\chi,\gamma)$
is a $\mathcal D$-subsolution whenever
\[
e^\gamma\geq
\left[\eta^{-1}\bigl(\psi-G_x(B)\bigr)_+\right]^{n/d}.
\]
In particular, the reference form need not belong to $\mathcal A$.
\end{cor}
\begin{proof}
For every $P\geq0$ in the test of Definition~\ref{def:D-subsolution},
\[
\eta\det_\rho(P)^{d/n}
\leq G_x(B+P)-G_x(B)\leq\psi-G_x(B),
\]
by \eqref{eq:conic-taylor-increment}.
\end{proof}

\subsection{Conic and polarized spectral examples}
\label{subsec:K-Garding-examples}

We first recall the conic Hessian examples and then verify the
analytic properties of polarized spectral operators directly.
We conclude with conic examples depending on fixed Hermitian forms. With respect to the reference form $\rho$ fixed
above, we write
$$
\sigma_j(A):=\sigma_j\bigl(\lambda_\rho(A)\bigr)
=\binom{n}{j}\frac{A^j\wedge\rho^{n-j}}{\rho^n},
\qquad 0\leq j\leq n.
$$

\begin{example}
\label{example:sigma-k-K-Garding}
\label{example:determinant-K-Garding}
For $1\leq k\leq n$, the polynomial $G(A)=\sigma_k(A)$ is
G{\aa}rding elliptic, with distinguished component
$$
\mathcal C_{\sigma_k}
=\left\{A\in Q(V):\sigma_j(A)>0,\quad 1\leq j\leq k\right\}.
$$
This follows from the classical theory of hyperbolic polynomials.
The cases $k=1$ and $k=n$ give
$$
\sigma_1(A)=\operatorname{tr}(\rho^{-1}A),
\qquad
\sigma_n(A)=\det_\rho A.
$$
In particular, $\mathcal C_{\sigma_n}$ is the cone of positive
definite Hermitian forms.
\end{example}

\begin{prop}[Polarized spectral admissible pairs]
\label{example:polarized-spectral-operators}
\label{thm:spectral-lift-polarization}
Let $p(t)=\sum_{j=0}^d a_jt^j$ be a univariate G{\aa}rding
polynomial, where $1\leq d\leq n$ and $a_d>0$, and set
\[
f(\x)=\pol_n p(\x)=\sum_{j=0}^d\frac{a_j}{\binom nj}\sigma_j(\x),
\qquad G_p(A)=f(\lambda_\rho(A)).
\]
Define the ordinary first-derived region and the spectral region by
\[
\mathcal C_f^{(1)}:=\bigcap_{i=1}^n\mathcal C_{\partial_i f},
\qquad
\Gamma_f:=
\begin{cases}
\mathcal C_f^{(1)},&d\geq2,\\
\mathcal C_f,&d=1,
\end{cases}
\qquad
\mathcal A_p:=\{A:\lambda_\rho(A)\in\Gamma_f\}.
\]
Then $(G_p,\mathcal A_p)$ is admissible and satisfies QSH. Whenever
$\lambda_\rho(A)\in\mathcal C_f^{(1)}$ and $P\geq0$,
\begin{equation}\label{eq:polarized-spectral-increment}
G_p(A+P)-G_p(A)
\geq\frac{a_d}{\binom nd}\sigma_d(P)
\geq a_d\det_\rho(P)^{d/n}.
\end{equation}
\end{prop}
\begin{proof}
Polarization makes $f$ symmetric, multi-affine, and G{\aa}rding
\cite{fang2026idealgaardingpolynomials}. The finite intersection
$\mathcal C_f^{(1)}$ is open and contains $\mathcal C_f$.
Weyl monotonicity makes $\mathcal A_p$ upward closed; increasing
segments to a common large scalar matrix show it is connected.
For $d\geq2$, if $\lambda_0\in\partial\mathcal C_f^{(1)}$ and
$f(\lambda_0)>0$, nonnegative Taylor coefficients give
$f(\lambda_0+t\one)>0$ for $t\geq0$. This ray eventually enters
$\mathcal C_f$, so $\lambda_0\in\mathcal C_f$, a contradiction.
Thus $f\leq0$ on that boundary, and $G_p\leq0$ on
$\partial\mathcal A_p$. For $d=1$, the boundary value is zero.

In a $\rho$-unitary basis diagonalizing $A$, put
$\lambda=\lambda_\rho(A)$. Expanding principal minors gives
\begin{equation}\label{eq:spectral-principal-minor-expansion}
G_p(A+P)-G_p(A)
=\sum_{\varnothing\ne S\subseteq[n]}
\det(P_{S,S})\,\partial_Sf(\lambda),
\qquad \partial_S=\prod_{i\in S}\partial_i.
\end{equation}
For $\lambda\in\mathcal C_f^{(1)}$, component nesting from the
first derivatives makes every term nonnegative. The terms with
$|S|=d$ sum to $a_d\sigma_d(P)/\binom nd$. Maclaurin's inequality
proves \eqref{eq:polarized-spectral-increment} and strict ellipticity.
For $d\geq2$, derivative nesting, and for $d=1$, $f(\lambda)>0$, give
\[
0<\partial_{\one}^{d-1}f(\lambda)
=(d-1)!\left(\frac{da_d}{n}\operatorname{tr}_\rho A+a_{d-1}\right).
\]
Hence QSH holds with $b_{\mathcal A_p}=\max\{0,na_{d-1}/(da_d)\}$.

\end{proof}

\begin{example}
The polynomial $p(t)=t^3-3t-3$ has one real root $r>2$, whereas
the largest roots of $p'$ and $p''$ are $1$ and $0$.
Thus it is G{\aa}rding, and Proposition~\ref{thm:spectral-lift-polarization}
applies, for $n\geq3$, to
\[
G_p(A)=\frac{\sigma_3(A)}{\binom n3}-\frac3n\sigma_1(A)-3.
\]
\end{example}

\begin{example}[The supercritical dHYM/LYZ polynomial]
\label{example:dHYM-Garding}
For $n\geq2$ and $0<\theta<\pi$, the polynomial
$p_\theta(t)=-\operatorname{Im}(e^{-i\theta}(t+i)^n)$ has leading
coefficient $\sin\theta>0$ and roots
$\cot((\theta+\ell\pi)/n)$, $0\leq\ell<n$. Its Hermitian lift is
\[
G_\theta(A)=-\operatorname{Im}\bigl(e^{-i\theta}
\det(\rho^{-1}A+iI)\bigr)
=\sum_{j=0}^n\sin\left(\theta-\frac{(n-j)\pi}{2}\right)\sigma_j(A).
\]
With $\operatorname{arccot}\in(0,\pi)$ and
$\lambda=\lambda_\rho(A)$, put
\[
\Theta_\rho(A)=\sum_j\operatorname{arccot}\lambda_j,
\qquad
\mathcal A_\theta=
\left\{A:\max_i\sum_{j\ne i}\operatorname{arccot}\lambda_j<\theta\right\}.
\]
The identity
\[
G_\theta(A)=\prod_j\sqrt{1+\lambda_j^2}\,
\sin(\theta-\Theta_\rho(A))
\]
identifies its distinguished positivity component as
$\{\Theta_\rho<\theta\}$. Applying the same identity to each
coordinate derivative identifies $\mathcal A_\theta$ with the
spectral first-derived region. Proposition~\ref{thm:spectral-lift-polarization}
gives an admissible QSH pair and
\[
G_\theta(A+P)-G_\theta(A)\geq\sin\theta\,\det_\rho P
\qquad(A\in\mathcal A_\theta,\ P\geq0).
\]
The phase level $\Theta_\rho=\theta$ lies in $\mathcal A_\theta$
but has $G_\theta=0$. Since \eqref{eq:main-equation} requires
$\psi>0$, we use the direct determinant tests of
Section~\ref{sec:examples-applications}.
\end{example}

\begin{rem}
\label{rem:dHYM-critical-phase}
At $\theta_{\mathrm{crit}}=\pi$, the degree of
$p_\theta$ drops to $n-1$. The polynomial remains real-rooted,
and its Hermitian lift is
$$
G_{\theta_{\mathrm{crit}}}(A)
=\sum_{r=0}^{\lfloor(n-1)/2\rfloor}
(-1)^r\sigma_{n-1-2r}(A).
$$
This is the polynomial associated with the critical LYZ equation
studied by Fu--Yau--Zhang \cite{fu2026criticallyzequationkahler}.
For $n\geq3$, its spectral first-derived region is
$\mathcal A_\pi=\{A:\max_i\sum_{j\ne i}\operatorname{arccot}\lambda_j<\pi\}$.
The critical phase $\Theta_\rho=\pi$ lies in this open region.

\end{rem}

\begin{example}[Mixed-form operators]
\label{example:mixed-form-operators}
Fix $1\leq k\leq n$ and positive definite Hermitian forms
$B_1,\ldots,B_{n-k}$. Then
$$
G_{B_1,\ldots,B_{n-k}}(A)
:=\frac{A^k\wedge B_1\wedge\cdots\wedge B_{n-k}}{\rho^n}
=\frac{k!}{n!}\,
\pdv_{B_1}\cdots\pdv_{B_{n-k}}\det_\rho(A)
$$
is G{\aa}rding elliptic by closure under positive directional
derivatives. The same conclusion holds for semipositive $B_i$
whenever the resulting polynomial is nonzero. These operators
need not be spectral with respect to $\rho$. Taking
$B_1=\cdots=B_{n-k}=\rho$ recovers
$\binom{n}{k}^{-1}\sigma_k$.
\end{example}

\begin{example}[Products of partial traces]
\label{example:products-K-Garding}
Consider a $\rho$-orthogonal splitting
$$
V=E\oplus F,\qquad
\dim_{\mathbb C}E=m\geq1,\quad
\dim_{\mathbb C}F=k\geq1,\quad m+k=n.
$$
Let $\pi_E,\pi_F$ be the corresponding projections, and write
$$
\operatorname{tr}_E A:=\operatorname{tr}(\pi_E\rho^{-1}A),
\qquad
\operatorname{tr}_F A:=\operatorname{tr}(\pi_F\rho^{-1}A).
$$
For positive integers $a,b$, set
$$
H(A):=
\left(\frac{\operatorname{tr}_E A}{m}\right)^a
\left(\frac{\operatorname{tr}_F A}{k}\right)^b.
$$

The map $A\mapsto(\operatorname{tr}_E A/m,\operatorname{tr}_F A/k)$
is surjective and sends semipositive forms to the nonnegative orthant.
Lemma~\ref{lem:positive-linear-pullback}, applied to $x^ay^b$,
shows that $H$ is G{\aa}rding elliptic.
Its distinguished component is
$$
\mathcal C_H
=\{A\in Q(V):\operatorname{tr}_E A>0,\
                  \operatorname{tr}_F A>0\}.
$$
\end{example}
\begin{rem}
\label{rem:monotonicity-direction-cone}
If $\mathcal K_1\subseteq\mathcal K_2$ are closed convex cones
with nonempty interior, every $\mathcal K_2$-G{\aa}rding
polynomial is $\mathcal K_1$-G{\aa}rding. Positive directional
differentiation preserves the direction cone, although it lowers
the degree of the polynomial.
\end{rem}

\section{Det-ellipticity and determinantal capacity}
\label{sec:det-ellipticityanddet}
The preceding results reduce the construction of $\mathcal D$-subsolutions to the det-ellipticity of the admissible pair. It is therefore natural to ask which fully nonlinear operators are det-elliptic. For a G{\aa}rding elliptic polynomial of degree $d\geq1$, we
characterize the determinant increment bound with exponent $d/n$.
Throughout, the top homogeneous part is G{\aa}rding elliptic by
Proposition~\ref{prop:conic-top-part}.

Let \(\R_+[\x]\) denote the cone of polynomials in
\(\x=(x_1,\ldots,x_n)\) with nonnegative coefficients. For \(\alpha\in\N^n\), we write \([\x^\alpha]f\) for the coefficient
of the monomial \(\x^\alpha\) in \(f\). The superscript $\mathfrak{h}$ denotes homogeneity.

\subsection{\(\beta\)-capacity and slope stability for polymatroid} \label{subsec:capacity-polymatroid-stability}
In this subsection, we study capacity for polynomials and its relation with slope stability of polymatroid.

The polynomial capacity was introduced by Gurvits \cite{MR2200852}.
Here we use its weighted version defined by Anari-Oveis Gharan and Gurvits-Leake \cite{AnariGharan21,GurvitsLeake21}.

\begin{defn}[\(\beta\)-capacity]
\label{def:beta-capacity}
Let \(0\neq f\in\R_+[\x]\) and
\(\beta\in\overline{\Gamma_n^+}\). The \emph{\(\beta\)-capacity} of
\(f\) is
\[
\Capacity_\beta(f)
:=
\inf_{\x\in\Gamma_n^+}
\frac{f(\x)}{\x^\beta},
\qquad
\x^\beta:=\prod_{i=1}^n x_i^{\beta_i}.
\]
\end{defn}
\begin{defn}[Newton polytope]
\label{def:newton-polytope}
The support of \(0\neq f\in\R_+[\x]\) is
\[
\operatorname{supp}(f)
:=
\left\{
\alpha\in\N^n:[\x^\alpha]f>0
\right\},
\]
and the associated \emph{Newton polytope} is
\[
\Newt(f)
:=
\operatorname{conv}\bigl(\operatorname{supp}(f)\bigr)
\subseteq
\overline{\Gamma_n^+}.
\]
\end{defn}

Anari-Oveis Gharan gave  a characterization of $f$ with positive $\beta$-capacity. See \cite[Fact 2.18]{AnariGharan21} for the proof.
\begin{lem}
\label{lem:beta-capacity-newton}
For \(0\neq f\in\R_+[\x]\) and
\(\beta\in\overline{\Gamma_n^+}\),
\[
\Capacity_\beta(f)>0
\quad\Longleftrightarrow\quad
\beta\in\Newt(f).
\]
\end{lem}

Next, we interpret positivity of the \(\beta\)-capacity as a polymatroid
slope-semistability condition. Let \(E:=[n]\), and, for
\(\alpha\in\mathbb R^E\) and \(S\subseteq E\), write
\[
\alpha(S):=\sum_{i\in S}\alpha_i.
\]
\begin{defn}[Polymatroid]
\label{def:integral-polymatroid}
An (integral) \emph{polymatroid} on \(E\) is a pair
\(\mathcal P=(E,r)\), where the rank function
\(r\colon 2^E\to\mathbb N\) satisfies:
\begin{enumerate}
    \item \(r(\varnothing)=0\);
    \item {Monotonicity:} if \(S\subseteq T\), then \(r(S)\leq r(T)\);
    \item {Submodularity:}
    $
    r(S)+r(T)\geq r(S\cap T)+r(S\cup T)$
    for all \(S,T\subseteq E\).
\end{enumerate}
The associated \emph{base polytope} is
\[
B(r)
:=
\left\{
\gamma\in\R_+^E:
\sum_{i\in S} \gamma_i\leq r(S)\ \text{for every }S\subseteq E,\
\sum_{i\in E} \gamma_i=r(E)
\right\}.
\]
\end{defn}

We introduce the weighted slope stability for  $\mathcal P$.
\begin{defn}[\(\beta\)-weighted slope semistability]
\label{def:beta-polymatroid-semistability}
Let $\mathcal P=(E,r)$ have total rank $d:=r(E)>0$, and let \(\beta\in\overline{\Gamma_n^+}\) satisfy \(\beta(E)=d\).
For \(S\subseteq E\) with \(r(S)>0\), define the
\emph{\(\beta\)-weighted slope} by
\[
\mu_\beta(S):=\frac{\beta(S)}{r(S)}.
\]
We say that \(\mathcal P\) is
\emph{\(\beta\)-weighted slope semistable} if
\(\mu_\beta(S)\leq\mu_\beta(E)=1\) for every \(S\subseteq E\) with
\(r(S)>0\), and \(\beta(S)=0\) whenever \(r(S)=0\).
\end{defn}

We use the standard correspondence between \(M\)-convex sets and
the integral points of integral base polyhedra; see
\cite[Chapter~4]{Murota03} and
\cite[Chapter~3]{MR2171629}.
\begin{lem}[Rank function of an \(M\)-convex set]
\label{lem:M-convex-support-rank}
Let \(B\subseteq\N^E\) be a nonempty \(M\)-convex set and define
\[
r_B(S):=\max_{\alpha\in B}\alpha(S),
\qquad S\subseteq E.
\]
Then \(r_B\) is the rank function of a polymatroid and
\[
\operatorname{conv}(B)=B(r_B),
\qquad
B=B(r_B)\cap\mathbb Z^E.
\]
\end{lem}

For the remainder of this subsection, let
\begin{equation}
0\neq f\in\R_+[x_1,\cdots,x_n]
\label{eq:M-convex-polynomial}
\end{equation}
be homogeneous of degree \(d\) with \(M\)-convex support.
By Lemma~\ref{lem:M-convex-support-rank},
\[
\mathcal P_f:=(E,r_f), \qquad r_f(S)
:=
\max_{\alpha\in\operatorname{supp}(f)}\alpha(S),
\] defines a polymatroid on $E$.
Moreover,
\begin{equation}
r_f(E)=d,
\qquad
\Newt(f)=B(r_f).
\label{eq:support-polymatroid-base}
\end{equation}

\begin{thm}
\label{thm:capacity-polymatroid-stability}
Let \(f\) satisfy \eqref{eq:M-convex-polynomial}, and let
\(\beta\in\overline{\Gamma_n^+}\) satisfy \(\beta(E)=d\). Then the
following are equivalent:
\begin{enumerate}[label=\textnormal{(\arabic*)}]
\item \(\Capacity_\beta(f)>0\).
\item \(\beta\in\Newt(f)\).
\item \(\beta\in B(r_f)\).
\item \(\beta(S)\leq r_f(S)\) for every \(S\subseteq E\).
\item \(\mathcal P_f\) is
\(\beta\)-weighted slope semistable.
\end{enumerate}
\end{thm}
\begin{proof}
By Lemma~\ref{lem:beta-capacity-newton},
\[
\Capacity_\beta(f)>0
\quad\Longleftrightarrow\quad
\beta\in\Newt(f).
\]
By \eqref{eq:support-polymatroid-base},
\(\Newt(f)=B(r_f)\). Since \(\beta(E)=r_f(E)=d\), membership in
\(B(r_f)\) is equivalent to
\(\beta(S)\leq r_f(S)\) for every \(S\subseteq E\).  This is precisely
\(\beta\)-weighted slope semistability.
\end{proof}
\begin{cor}[Balanced slope criterion]
\label{cor:balanced-polymatroid-stability}
Under the assumptions of
Theorem~\ref{thm:capacity-polymatroid-stability}, set
\(\beta=\frac{d}{n}\one\). Then for every \(S \subseteq E\) with \(r_f(S)>0\), \[ \beta(S)\leq r_f(S) \quad\Longleftrightarrow\quad \frac{d}{n}|S|\leq r_f(S) \quad\Longleftrightarrow\quad \frac{|S|}{r_f(S)}\leq\frac{n}{d}. \] \end{cor}
\begin{proof}
Since \(\beta(S)=\frac{d}{n}|S|\), the result follows directly from
Theorem~\ref{thm:capacity-polymatroid-stability}.
\end{proof}

\subsection{The subspace polymatroid}
\label{subsec:subspace-polymatroid}
We now give an intrinsic formulation of the polymatroid ranks arising
from unitary scalarizations of \(H\). We use the rank-function
language of polymatroids; see
\cite{MR270945,MR2171629,MR1956924,MR1956925,MR1956926}.
\begin{defn}[Subspace polymatroid]
\label{def:abstract-subspace-polymatroid}
Let \(V\) be a finite-dimensional complex vector space of dimension $n$ and let
\(\mathcal L(V)\) denote its lattice of complex subspaces. A
\emph{subspace polymatroid} on \(V\) is a pair
\(\mathcal P=(V,r)\), where
\(r\colon\mathcal L(V)\to\mathbb R_{\geq 0}\) satisfies: for all \(U,W\in\mathcal L(V)\),
\begin{enumerate}
    \item \(r(\{0\})=0\);
    \item {Monotonicity:} if \(U\subseteq W\), then \(r(U)\leq r(W)\);
    \item {Submodularity:}
    $
    r(U)+r(W)\geq r(U\cap W)+r(U+W)$
\end{enumerate}
It is called \emph{integral} if \(r\) is integer-valued.
\end{defn}

Let $H\in\mathbb R[Q(V)]$ be a homogeneous $Q(V)^+$-G{\aa}rding
polynomial of degree $d\geq1$. For a complex subspace \(U\subseteq V\),
let \(\pi_U\in\overline{Q(V)^+}\) denote the
\(\rho\)-orthogonal projection onto \(U\).
\begin{defn}[Subspace rank]
\label{def:subspace-rank}Notations as above.
For \(U\in\mathcal L(V)\), define
\begin{equation}
r_H(U):=\deg_t H(I+t\pi_U).
\label{eq:subspace-rank}
\end{equation}
\end{defn}
Since \(H\) is homogeneous of degree \(d\), one has
\(r_H(0)=0\) and \(r_H(V)=d\).
\begin{lem}
\label{lem:directional-degree-image}
Let \(A,B\in\overline{Q(V)^+}\). If
\(\operatorname{im}A=\operatorname{im}B\), then
\[
\deg_t H(I+tA)=\deg_t H(I+tB).
\]
Consequently, if \(U=\operatorname{im}A\), then
\(\deg_t H(I+tA)=r_H(U)\).
\end{lem}
\begin{proof}
First suppose that \(0\leq A_1\leq A_2\), and write
\(A_2=A_1+C\) with \(C\geq 0\). The scalarization
\[
q(x,y,z):=H(xI+yA_1+zC)
\]
is a homogeneous G{\aa}rding polynomial with nonnegative
coefficients. Hence
\[
\deg_t H(I+tA_1)
=
\deg_t q(1,t,0)
\leq
\deg_t q(1,t,t)
=
\deg_t H(I+tA_2).
\]
Now suppose that \(\operatorname{im}A=\operatorname{im}B\).
Since \(A\) and \(B\) are positive definite on their common image
and vanish on its orthogonal complement, there exist \(c,C>0\) such
that \(cB\leq A\leq CB\). The preceding monotonicity, together with
invariance of the degree under positive rescaling of the direction,
gives
\[
\deg_t H(I+tB)=\deg_t H(I+tA).
\]
Taking \(B=\pi_U\) proves the final assertion.
\end{proof}
\begin{thm}
\label{thm:subspace-rank-submodular}
With $r_H$ in Definition~\ref{def:subspace-rank},
\(\mathcal P_H:=(V,r_H)\) is an integral subspace polymatroid.
\end{thm}
\begin{proof}
$r_H(\{0\})=0$ and the monotonicity of $r_H$ is obvious. So it remains to prove the submodularity.

Let $U,W\subseteq V$, set $Z:=U\cap W$, and write
$
U=Z\oplus U_0,$ $ W=Z\oplus W_0,$
where both sums are $\rho$-orthogonal. Write $$A_1=I,\quad A_2=\pi_{U_0},\quad A_3=\pi_{W_0},\quad A_4=\pi_{Z}$$
By Corollary \ref{cor:homogeneous-K-scalarization}, the polynomial
\[
q(s_1,s_2,s_3,s_4):=H(\sum_{i=1}^4 s_i A_i),
\]
is homogeneous \gar{} polynomial of degree $d:=\deg H$, and hence Lorentzian by \cite[Theorem 11.2]{Fang-MaGard}. Then, \cite[Theorem 2.23]{BrandenHuh20} shows that
$\operatorname{supp}q$ is $M$-convex and its rank function
\[
\varrho(S):=\max_{\alpha\in\operatorname{supp}q}\sum_{i\in S}\alpha_i,
\qquad S\subseteq\{1,2,3,4\},
\]
is submodular.

We claim that
\begin{equation}
\varrho(S)
=
\deg_t H\left(I+t\sum_{i\in S}A_i\right),
\qquad S\subseteq\{2,3,4\}.
\label{eq:support-rank-directional-degree}
\end{equation}
Let $e_1,\ldots,e_4$ denote the standard basis of $\mathbb R^4$.
Since
\[
q(s_1,0,0,0)=H(I)s_1^d
\quad\text{and}\quad H(I)>0,
\]
we have $d e_1\in\operatorname{supp}q$.
Choose $\alpha\in\operatorname{supp}q$ attaining $\varrho(S)$.
If $\alpha_j>0$ for some $j\notin S\cup\{1\}$, apply the
$M$-convex exchange property to $\alpha$ and $d e_1$ at
the coordinate $j$. The only coordinate at which $\alpha$
can be smaller than $d e_1$ is the first, so
\[
\alpha-e_j+e_1\in\operatorname{supp}q.
\]
Repeating this operation eliminates all coordinates outside
$S\cup\{1\}$ while preserving $\sum_{i\in S}\alpha_i$.
Thus a monomial attaining $\varrho(S)$ survives the
specialization
\[
s_1=1,\qquad
s_i=
\begin{cases}
t,&i\in S,\\
0,&i\in\{2,3,4\}\setminus S.
\end{cases}
\]
This proves \eqref{eq:support-rank-directional-degree}.

Submodularity applied to $\{2,4\}$ and $\{3,4\}$ gives
\begin{equation}
\varrho(\{2,4\})+\varrho(\{3,4\})
\geq
\varrho(\{2,3,4\})+\varrho(\{4\}).
\label{eq:directional-rank-submodularity}
\end{equation}
Now
\[
A_2+A_4=\pi_U,\quad
A_3+A_4=\pi_W,\quad
A_4=\pi_{U\cap W}, \quad \mathrm{im} \left(\sum_{i=2}^4 A_i\right)=U+W.
\]
Combining \eqref{eq:support-rank-directional-degree} with
Lemma~\ref{lem:directional-degree-image}, we identify
\eqref{eq:directional-rank-submodularity} with
\[
r_H(U)+r_H(W)
\geq
r_H(U+W)+r_H(U\cap W).
\]
Thus $r_H$ is submodular. If \(U\subseteq W\), then \(\pi_U\leq\pi_W\), so the monotonicity
proved in Lemma~\ref{lem:directional-degree-image} gives
\(r_H(U)\leq r_H(W)\).
\end{proof}
We next connect the intrinsic rank \(r_H\) with unitary
scalarizations. Let   \( \v=\{v_1,\ldots,v_n\}\) be a \(\rho\)-unitary basis of \(V\). Set
\(E_i^{\mathbf v}:=v_i\otimes\overline{v_i}\), and define
\begin{equation}
g_{\mathbf v}(\mathbf x)
:=
H\left(\sum_{i=1}^n x_iE_i^{\mathbf v}\right).
\label{eq:unitary-scalarization-H}
\end{equation}
By Corollary~\ref{cor:homogeneous-K-scalarization},
\(g_{\mathbf v}\) is homogeneous of degree \(d\), G{\aa}rding, and
has nonnegative coefficients. Therefore, by \cite{Fang-MaGard}, it is Lorentzian and
\(\operatorname{supp}(g_{\mathbf v})\) is \(M\)-convex. Denote
the rank function of its support polymatroid by
\[
r_{\mathbf v}(S)
:=
\max_{\alpha\in\operatorname{supp}(g_{\mathbf v})}\alpha(S),
\qquad
S\subseteq[n].
\]
\begin{lem}
\label{lem:adapted-frame-rank}
Suppose that \(\mathbf v\) is adapted to \(U\subseteq V\), so that
\(U=\operatorname{span}_{\mathbb C}\{v_i:i\in S\}\) for some
\(S\subseteq[n]\). Then
\[
r_H(U)=r_{\mathbf v}(S).
\]
\end{lem}
\begin{proof}
Since \(I=\sum_{i=1}^nE_i^{\mathbf v}\) and
\(\pi_U=\sum_{i\in S}E_i^{\mathbf v}\),
\[
H(I+t\pi_U)
=
g_{\mathbf v}\bigl((1+t)\one_S+\one_{S^c}\bigr).
\]
Writing
\(g_{\mathbf v}(\mathbf x)
=\sum_\alpha c_\alpha(\mathbf v)\mathbf x^\alpha\), with
\(c_\alpha(\mathbf v)\geq 0\), gives
\[
H(I+t\pi_U)
=
\sum_\alpha
c_\alpha(\mathbf v)(1+t)^{\alpha(S)}.
\]
There is no cancellation in the highest \(t\)-degree, and hence
\[
r_H(U)
=
\max_{\alpha\in\operatorname{supp}(g_{\mathbf v})}\alpha(S)
=
r_{\mathbf v}(S).
\]
\end{proof}
\begin{defn}[Subspace slope semistability]
\label{def:subspace-slope-semistability}
For a nonzero complex subspace \(U\subseteq V\), define
\[
\mu_H(U):=\frac{\dim_{\mathbb C}U}{r_H(U)},
\]
with the convention \(\mu_H(U)=+\infty\) if \(r_H(U)=0\). We call
\(\mathcal P_H\) \emph{slope semistable} if
\(\mu_H(U)\leq\mu_H(V)=n/d\) for every nonzero proper complex
subspace \(U\subsetneq V\).
\end{defn}
The following theorem relates the intrinsic subspace polymatroid to
the Newton polytopes and capacities of all unitary scalarizations.
\begin{thm}
\label{thm:capacity-subspace-slope-stability}
Set
\(
\beta:=\frac{d}{n}\one\in\mathbb R^n.\) Then the following
statements are equivalent:
\begin{enumerate}[label=\textnormal{(\arabic*)}]
\item\label{enu:subspace-slope-1} The subspace polymatroid \(\mathcal P_H\) is slope semistable.
\item\label{enu:subspace-slope-2} For every complex subspace \(U\subseteq V\),
\(
r_H(U)\geq\frac{d}{n}\dim_{\mathbb C}U.
\)
\item\label{enu:subspace-slope-3} For every \(\rho\)-unitary basis \(\mathbf v\) and every
\(S\subseteq[n]\),
\(
r_{\mathbf v}(S)\geq\frac{d}{n}|S|.
\)
\item\label{enu:subspace-slope-4} For every \(\rho\)-unitary basis \(\mathbf v\),
\(
\beta\in\Newt(g_{\mathbf v}).
\)
\item\label{enu:subspace-slope-5} For every \(\rho\)-unitary basis \(\mathbf v\),
\(
\Capacity_\beta(g_{\mathbf v})>0.
\)
\item\label{enu:subspace-slope-6} There exists  $c>0$ such that for any \(\rho\)-unitary basis \(\mathbf v\),
\(
\Capacity_\beta(g_{\mathbf v})>c.
\)
\end{enumerate}
\end{thm}
\begin{proof}
The equivalence of \ref{enu:subspace-slope-1} and \ref{enu:subspace-slope-2} is the
definition of slope semistability.
Assume \ref{enu:subspace-slope-2}. For a \(\rho\)-unitary basis \(\mathbf v\)
and \(S\subseteq[n]\), apply \ref{enu:subspace-slope-2} to
\(U_S:=\operatorname{span}_{\mathbb C}\{v_i:i\in S\}\).
Lemma~\ref{lem:adapted-frame-rank} gives
\[
r_{\mathbf v}(S)
=
r_H(U_S)
\geq
\frac{d}{n}\dim_{\mathbb C}U_S
=
\frac{d}{n}|S|,
\]
so \ref{enu:subspace-slope-3} holds. Conversely, every complex subspace admits
an adapted \(\rho\)-unitary basis, so \ref{enu:subspace-slope-3} and
Lemma~\ref{lem:adapted-frame-rank} imply \ref{enu:subspace-slope-2}.
Since
\[
\Newt(g_{\mathbf v})=B(r_{\mathbf v})
\qquad\text{and}\qquad
\beta([n])=d,
\]
membership of \(\beta\) in \(\Newt(g_{\mathbf v})\) is equivalent to
\[
\beta(S)=\frac{d}{n}|S|\leq r_{\mathbf v}(S)
\qquad
(S\subseteq[n]).
\]
Hence \ref{enu:subspace-slope-3} and \ref{enu:subspace-slope-4} are equivalent.
Finally, Lemma~\ref{lem:beta-capacity-newton} gives
\[
\Capacity_\beta(g_{\mathbf v})>0
\quad\Longleftrightarrow\quad
 \beta\in\Newt(g_{\mathbf v}),
\]
which proves the equivalence of \ref{enu:subspace-slope-4} and
\ref{enu:subspace-slope-5}.

Finally,  $g_\v^n$ is also Lorentzian with  $$\Newt(g_\v^n)=\sum_{i=1}^n\Newt(g_\v)$$ under the Minkowski sum. Since $\supp g_\v^n$  is also  $M$-convex, if $\beta\in \Newt(g_{\v})$, then $$n\beta=d \mathbf{1}\in\Newt(g_\v^n)\cap \mathbb Z^n=\supp g_\v^n.$$  Thus,  $$(\Capacity_{\beta} (g_\v))^n\geq a_\v:=[\x^{d\mathbf{1}}]g_{\v}^n>0.$$ Since the coefficients of $g_\v$ are continuous functions of $\v$, and $U(n)$ is compact, we have $a_\v>c$ for some $c>0$ for all $\v$. This establishes \ref{enu:subspace-slope-4} $\Rightarrow$
\ref{enu:subspace-slope-6}. \ref{enu:subspace-slope-6} clearly implies \ref{enu:subspace-slope-5}. So we have finished the proof.
\end{proof}

\subsection{Determinantal Capacity and slope semistability}
We define the following capacity on the real vector space $Q(V)$.

\begin{defn}[Determinantal capacity]\label{det-capacity}
Let
\(
H\in \R[Q(V)]\)
be a homogeneous polynomial of degree $d\geq1$ that is positive on
\(Q(V)^{++}\). The \emph{determinantal capacity} of \(H\) is defined by
\[
\Capacity_{\det}(H)
:=
\inf_{P\in Q(V)^{++}}
\frac{H(P)}{\det_\rho(P)^{d/n}}.
\]
\end{defn}

\begin{thm}
\label{thm:det-ellipticity-capacity}
Let \(G\) be a \gar{} elliptic operator of degree $d\geq1$, \(H:=G^{\mathrm{top}}\), {$\mathcal A:=\operatorname{int}\mathcal C_G^{(1)}$}, and let \(\mathcal P_H=(V,r_H)\) be the associated subspace polymatroid. Let $\beta=\frac{d}{n}\mathrm{1}$. Then the following are equivalent:
\begin{enumerate}[label=\textnormal{(\arabic*)}]
\item\label{enu:det-ellipticity-1} \(\exists\,\eta>0\) s.t. \(\forall A\in\mathcal A,\;P\in Q(V)^+,\)
\[
G(A+P)-G(A)\;\ge\;\eta\,\det_\rho(P)^{d/n}
\]
\item\label{enu:det-ellipticity-2} \(\exists\,\eta>0\) s.t. \(P\in Q(V)^+\),
\[
H(P)\;\ge\;\eta\,\det_\rho(P)^{d/n}
\]
\item\label{enu:det-ellipticity-3} \(\Capacity_{\det}(H)>0\).
\item\label{enu:det-ellipticity-4} For any unitary scalarization  $g_\v$ defined in \eqref{eq:unitary-scalarization-H}, $\Capacity_\beta(g_\v)>0$.
\item\label{enu:det-ellipticity-5} \(\mathcal P_H\) is slope semistable.
\item\label{enu:det-ellipticity-6} For every unitary scalarization \(h_\v\) of \(H\), the associated
polymatroid \(\mathcal P_{h_\v}\) is \(\beta\)-semistable.
\end{enumerate}
Moreover,
\[
\eta_{\max}=\Capacity_{\det}(H).
\]
The same increment bound holds for all $A\in\mathcal C_G^{(1)}$.

\end{thm}
\begin{proof}
\ref{enu:det-ellipticity-1}$\Rightarrow$\ref{enu:det-ellipticity-2}.
Fix \(A\in\mathcal A\), \(P\ge0\). Apply \ref{enu:det-ellipticity-1} to \(tP\):
\[
G(A+tP)-G(A)\ge \eta\,t^d\det_\rho(P)^{d/n}.
\]
Divide by \(t^d\), let \(t\to\infty\), we get \ref{enu:det-ellipticity-2}.

\ref{enu:det-ellipticity-2}$\Rightarrow$\ref{enu:det-ellipticity-1}.
For \(A\in\mathcal A\), \(P\ge0\), Taylor expansion in direction \(P\) gives
\[
G(A+P)-G(A)=\sum_{k=1}^d \frac{1}{k!}\partial_P^k G(A).
\]
{Since $A\in\mathcal A\subseteq\mathcal C_G^{(1)}$,} all nonzero \(\partial_P^k G(A)\ge0\), hence
\[
G(A+P)-G(A)\ge \frac{1}{d!}\partial_P^d G(A)=H(P).
\]
Combine with \ref{enu:det-ellipticity-2}, we get \ref{enu:det-ellipticity-1}. At the same time, we obtain the optimal constant claim.

\ref{enu:det-ellipticity-2}$\Leftrightarrow$\ref{enu:det-ellipticity-3}. It follows from Definition~\ref{det-capacity}.

\ref{enu:det-ellipticity-3}$\Leftrightarrow$\ref{enu:det-ellipticity-4}.
Unitary diagonalization gives
\[
\Capacity_{\det}(H)
=\inf_{\mathbf v\in U(n)}\Capacity_\beta(g_{\mathbf v}).
\]
Hence \ref{enu:det-ellipticity-3} implies \ref{enu:det-ellipticity-4}.
The converse follows from
\ref{enu:subspace-slope-5}$\Rightarrow$\ref{enu:subspace-slope-6} in
Theorem~\ref{thm:capacity-subspace-slope-stability}, which supplies
a uniform positive lower bound.

The rest of the claim follows from Theorem~\ref{thm:capacity-subspace-slope-stability}.
\end{proof}

\begin{example}
\label{rem:spectral-semistability}

The top parts in Example~\ref{example:sigma-k-K-Garding},
Proposition~\ref{thm:spectral-lift-polarization}, and
Example~\ref{example:dHYM-Garding} are positive multiples of
$\sigma_d$. Their subspace polymatroids are slope semistable
by the following symmetry argument.

More generally, let $G$ be a spectral G{\aa}rding elliptic operator
of degree $d\geq1$, and write
$$
H(A):=G^{\mathrm{top}}(A)=h\bigl(\lambda_\rho(A)\bigr).
$$
Then $h$ is homogeneous, symmetric, and has nonnegative
coefficients. If $\alpha\in\operatorname{supp}(h)$, every
coordinate permutation of $\alpha$ also belongs to the support.
Since $|\alpha|=d$,
$$
\frac{1}{n!}\sum_{\sigma\in\mathfrak S_n}\sigma\alpha
=\frac dn\one\in\Newt(h).
$$
Every unitary scalarization of $H$ equals $h$. Hence
Theorem~\ref{thm:det-ellipticity-capacity} implies that
$\mathcal P_H$ is slope semistable and $G$ is det-elliptic with exponent $d/n$.
\end{example}

\begin{example}
\label{example:partial-trace-capacity}
We return to the partial-trace product $H$ in
Example~\ref{example:products-K-Garding} and retain its
notation. Put $d=a+b$. By \eqref{eq:subspace-rank},
$$
r_H(E)=a,\qquad r_H(F)=b.
$$
Thus slope semistability requires
\[
a\geq\frac dn m,\quad  \text{ and } \quad  b\geq\frac dn k.
\]
Since $a+b=d$ and $m+k=n$, both inequalities must be equalities.
Consequently, the necessary balance condition is
\begin{equation}
    \label{1111}\frac am=\frac bk=\frac dn.
\end{equation}

Conversely, for the diagonal blocks $P_E,P_F$ of $P>0$,
the block determinant inequality and AM--GM give
\[
\det_\rho P\leq\det P_E\det P_F
\leq\left(\frac{\operatorname{tr}_E P}{m}\right)^m
\left(\frac{\operatorname{tr}_F P}{k}\right)^k.
\]
Under \eqref{1111}, raising to $d/n$ gives
$\det_\rho(P)^{d/n}\leq H(P)$, with equality at $I$;
hence $\Capacity_{\det}(H)=1$.
Otherwise, $P_t=tI_E\oplus t^{-m/k}I_F$ has determinant one and
$H(P_t)=t^{a-mb/k}\to0$ as $t\to0$ or $t\to\infty$,
so $\Capacity_{\det}(H)=0$.
\end{example}

\begin{rem}\label{rem:comparison-HL}
Our results extend the scope of the determinant-majorization theorem
of Harvey--Lawson \cite{HarveyLawsonMajorization} in three respects.
First, our framework accommodates nonhomogeneous G{\aa}rding
polynomials.

Second, even in the homogeneous setting, our G{\aa}rding framework
strictly extends the classical hyperbolic G{\aa}rding--Dirichlet
setting. For example, the polynomial
\[
f(x,y)=xy(x^2+cxy+y^2),\qquad \sqrt{3}\le c<2,
\]
is G{\aa}rding but not real stable
\cite[Example~11.4]{Fang-MaGard}.
Its matrix realization $H(A)=f(A_{11},A_{22})$ on real symmetric
$2\times2$ matrices is nonhyperbolic and has positive determinant
capacity.

Third, even within the classical hyperbolic setting, positive
determinant capacity is strictly weaker than the central-ray
hypothesis in \cite{HarveyLawsonMajorization} relative to the fixed background metric.
The latter implies that the capacity equals $H(I)$ and is attained
at $I$, whereas our criterion does not require attainment.
For instance,
\[
H(A)=A_{11}\operatorname{tr}A,\qquad A\in\operatorname{Sym}^2(\mathbb R^2),
\]
is hyperbolic with a G{\aa}rding cone containing the positive definite
cone, but its determinant capacity equals $1$ and is not attained.
In particular, no positive definite change of background makes
this operator satisfy the central-ray hypothesis.
\end{rem}

\section{Examples and Applications}
\label{sec:examples-applications}

We apply Theorem~\ref{thm:determinant-relative-Linfty} and its
corollaries to $J$-equations, deformed Hermitian Yang--Mills
equations, and inverse $\sigma_2$-equations.

We first give a relative estimate suited to semistable limits of the $J$-equation.

\begin{prop}
\label{prop:J-relative}
Let $(X,\rho)$ be compact K\"ahler of dimension $n\geq 2$, let
$\omega,\alpha$ be K\"ahler forms, and let $\chi$ be a smooth closed real
$(1,1)$-form with $[\chi]$ big. For $A>0$, write
$$
P_\alpha(A)=\mu_1+\cdots+\mu_{n-1},
$$
where $\mu_1\geq\cdots\geq\mu_n>0$ are the eigenvalues of $\alpha$
relative to $A$.

Let $\underline u\leq 0$ be continuous on $X\setminus Z$, with
$\underline u\to-\infty$ along a proper closed set $Z$. Assume that,
whenever $q$ touches $\underline u$ from above,
$$
B_q:=\omega+\ddc q-\chi\geq a\alpha,
\qquad
P_\alpha(B_q)\leq c-\eta
$$
for fixed $a,\eta>0$ and $c>0$.

Let $s_j,r_j\geq 0$ be uniformly bounded and let
$u_j\in C^\infty(X)$, $\sup_Xu_j=0$, solve
$$
\omega_j:=(1+s_j)\omega+\ddc u_j>0,
\qquad
\tr_{\omega_j}\alpha+r_j\det_{\omega_j}\alpha=c.
$$
Then there exists $C>0$, independent of $j$, such that
$$
u_j\geq \underline u+W_\chi-C
\qquad\text{on }X\setminus Z.
$$
In particular, if $\chi\geq 0$, then $W_\chi=0$ and
$$
u_j\geq \underline u-C
\qquad\text{on }X\setminus Z.
$$
\end{prop}

\begin{proof}
Set
$$
F_j(A):=
\left(\tr_A\alpha+r_j\det_A\alpha\right)^{-1},
\qquad A>0.
$$
Then $F_j(\omega_j)=1/c$.

We first claim that $\underline u$ is a viscosity $\mathcal D$-subsolution of the
$j$-th equation, with shift $\chi$ and a determinant bound independent
of $j$. Let $q$ be an upper test for $\underline u$ and let $P\geq 0$
satisfy
$$
A:=(1+s_j)\omega+\ddc q-\chi+P>0,
\qquad
F_j(A)\leq \frac1c.
$$
Since
$$
A=B_q+s_j\omega+P\geq B_q\geq a\alpha,
$$
if $\mu_1\geq\cdots\geq\mu_n>0$ are the eigenvalues of $\alpha$
relative to $A$, then
$$
\mu_1+\cdots+\mu_{n-1}
=
P_\alpha(A)
\leq
P_\alpha(B_q)
\leq c-\eta.
$$
Moreover,
$$
\mu_1+\cdots+\mu_n+r_j\mu_1\cdots\mu_n\geq c.
$$
Hence, writing $r_*:=\sup_j r_j$ and using $\mu_k\leq a^{-1}$,
$$
\eta
\leq
\mu_n\left(1+r_*a^{-(n-1)}\right).
$$
Thus $\mu_n\geq\delta>0$ uniformly in $j$, so $A\leq\delta^{-1}\alpha$.
Since
$$
0\leq P\leq A,
$$
we obtain
$$
\det_\rho P\leq C_{\det}
$$
for a constant $C_{\det}$ independent of $j$. Therefore
$(\underline u,\chi,\gamma)$ is a viscosity $\mathcal D$-subsolution for the
$j$-th equation with
$$
\gamma:=\log C_{\det}.
$$

Set
$$
U_j:=u_j-\underline u
$$
on $X\setminus Z$, and extend $U_j$ by $+\infty$ on $Z$.
Since $u_j$ is smooth, Proposition~\ref{prop:determinant-reduction} gives
$$
\det_{\rho,+}(\chi+\ddc q)\leq e^\gamma
$$
for every lower test $q$ of $U_j$.

We now apply Theorem~\ref{thm:determinant-relative-Linfty}, case~\ref{enu:thm-exp}. Since $\gamma$ is constant,
its entropy is uniformly bounded. Also,
$$
(W_\chi-U_j)_+
=
(W_\chi+\underline u-u_j)_+
\leq -u_j.
$$
If $s_*:=\sup_j s_j$, then $\omega_j>0$ gives
$$
\ddc u_j\geq -(1+s_*)\omega.
$$
Together with $\sup_Xu_j=0$, the Green-function estimate yields
$$
\sup_j\int_X(-u_j)\rho^n<\infty.
$$
Finally, $\chi$ is fixed and smooth, so Lemma~\ref{lem:uniform-alpha-integrability} gives the required
uniform exponential integrability. Hence Theorem~\ref{thm:determinant-relative-Linfty} gives
$$
\sup_X(W_\chi-U_j)\leq C,
$$
with $C$ independent of $j$. This is exactly
$$
u_j\geq \underline u+W_\chi-C
$$
on $X\setminus Z$. If $\chi\geq 0$, then $W_\chi=0$.
\end{proof}

\begin{example}
\label{ex:J-fourfold-null-curve}
We consider a semistable fourfold for the $J$-equation with a null curve.
Let
$$
X=\mathbb P_{\mathbb P^1}
\bigl(\mathcal O\oplus\mathcal O(-1)^{\oplus3}\bigr)
$$
in the lines convention, and write
$h=\pi^*c_1(\mathcal O_{\mathbb P^1}(1))$,
$\xi=c_1(\mathcal O_X(1))$, and
$C=\mathbb P_{\mathbb P^1}(\mathcal O)$.
The semiample class $\xi$ contracts exactly $C$. Set  K\"ahler classes 
$$
A=\xi+\frac1{12}h,\qquad
B=\frac3{10}\xi+\frac{13}{120}h,$$ and
$c=4\frac{B\cdot A^3}{A^4}=\frac{13}{10},
$
such that $cA-B=\xi$. Computation on
torus-invariant cycles shows that the numerical $J$-inequalities
are strict in dimensions two and three. Thus $(A,B)$ is
$J$-semistable, with $C$ its unique null curve.

The class
\begin{equation}\label{eq:fourfold-big-class}
\Theta=cA-3B
=\frac25\xi-\frac{13}{60}h
=\left(\frac3{20}\xi+\frac1{30}h\right)+\frac14(\xi-h)
\end{equation}
is big: the first summand is K\"ahler, and the linear system
$|\xi-h|$ has base locus $C$. Since $\Theta\cdot C=-13/60$,
we have $E_{\mathrm{nK}}(\Theta)=C$; in particular, $\Theta$ is not nef.

Fix arbitrary K\"ahler forms $\omega\in A$ and $\alpha\in B$.
By \eqref{eq:fourfold-big-class}, there is $\psi\leq0$, smooth
outside $C$ with analytic poles along $C$, such that
$$
R=c\omega-3\alpha+\ddc\psi\geq\delta\alpha
$$
for some $\delta>0$. For sufficiently small $\tau>0$, set
$b=(1-\tau)\psi/c$ and $\chi=\tau\omega$. Then, on $X\setminus C$,
$$
T:=\omega+\ddc b-\chi
=\frac{1-\tau}{c}(3\alpha+R)\geq a\alpha,
\qquad a>\frac3c.
$$

Put $\alpha_t=\alpha+t\omega$ and $c_t=c+4t$.
By the semistability of $(A,B)$, for every
$p$-dimensional subvariety $V$ with $1\leq p\leq3$,
\[
\bigl(c_tA^p-p(B+tA)A^{p-1}\bigr)\cdot V
\geq(4-p)tA^p\cdot V>0.
\]

The numerical criterion~\cite{MR4781476}
therefore gives smooth solutions

$$
\Omega_t=\omega+\ddc u_t>0,\qquad
\operatorname{tr}_{\Omega_t}\alpha_t=c_t,\qquad
\sup_Xu_t=0.
$$
For small $t>0$, the strict cone margin of $T$ relative to
$\alpha_t$ is uniform. The contact argument of
Proposition~\ref{prop:J-relative} therefore gives a uniformly
bounded determinant density for $u_t-b$, extended by $+\infty$
on $C$. Since $\chi$ is fixed and K\"ahler, the entropy bounds
are uniform; positivity of $\Omega_t$ gives the normalized $L^1$
bound. Theorem~\ref{thm:determinant-relative-Linfty} yields
\begin{equation}\label{eq:fourfold-relative-bound}
b-K\leq u_t\leq0\qquad\text{on }X\setminus C,
\end{equation}
and hence uniform local boundedness off $C$.

There is no smooth K\"ahler representative $\Omega\in A$ with
$c\Omega^3-3\alpha\wedge\Omega^2\geq0$: this would imply
$c\Omega-\alpha>0$, contradicting $(cA-B)\cdot C=0$.
Thus \eqref{eq:fourfold-relative-bound} uses a barrier singular
only along a codimension-three curve, without any symmetry
assumption on the background metrics.
\end{example}

\begin{example}
\label{ex:dHYM-critical-degeneration}
For dHYM/LYZ equation, we apply the critical-phase operator of
Remark~\ref{rem:dHYM-critical-phase} on
$X=\operatorname{Bl}_p\mathbb P^4$. Let $H$ and $E$ denote the
hyperplane pullback and exceptional divisor, and fix an arbitrary
K\"ahler form $\rho\in2H-E$. Let
$\eta=\pi^*\omega_{\mathrm{FS}}\in H$, where $\pi$ is the blowdown.
For $0<t\ll1$, set $r=1/\sqrt3$, $b_t=r+t$, and let $a_t$ be the
largest real root of
\[
2a_t^3-8a_t+b_t-b_t^3=0.
\]
Then $a_t\to a_0>2r$. The K\"ahler forms
$$
\beta_t:=b_t\rho+(a_t-2b_t)\eta\in a_tH-b_tE
$$
satisfy $\int_X(\rho+\sqrt{-1}\beta_t)^4<0$.
Since $\beta_t\geq b_t\rho$ and $3\arctan b_t>\pi/2$, they are
strict subsolutions at phase $\pi$. By
\cite[Theorem~1.1]{fu2026criticallyzequationkahler}, there are smooth solutions
$$
\alpha_t=\beta_t+\ddc u_t,\qquad
\sum_{i=1}^4\arctan\lambda_i^\rho(\alpha_t)=\pi,
\qquad \sup_Xu_t=0.
$$

Take
$$
\chi_t=\beta_t-r\rho
=t\rho+(a_t-2b_t)\eta
\longrightarrow(a_0-2r)\eta.
$$

If $q$ touches $u_t$ from below and $P=\chi_t+\ddc q>0$, then
$A:=r\rho+P\leq\alpha_t$. Since $A>r\rho$ and
$\Theta_\rho(\alpha_t)=\pi$, we have
$\pi\leq\Theta_\rho(A)<4\pi/3$, hence
$G_\pi(A)=\sigma_3^\rho(A)-\sigma_1^\rho(A)\leq0$.
At $r\rho$ the first derivatives vanish, so the reference lies on
the boundary of the first-derived region. Direct expansion gives
\[
\sigma_3^\rho(P)+\frac2{\sqrt3}\sigma_2^\rho(P)
\leq\frac8{3\sqrt3},\] \[
\det_\rho P\leq\left(1-\frac1{\sqrt3}\right)^4,
\]
where the determinant bound follows from Maclaurin's inequalities.

Thus the determinant density is uniformly bounded. The auxiliary
forms converge smoothly to a semipositive big form, while the phase
condition gives $\Delta_\rho u_t\geq-K$. The entropy and integrability
hypotheses of Theorem~\ref{thm:determinant-relative-Linfty} are therefore uniform,
and
$$
\sup_{0<t\ll1}\|u_t\|_{L^\infty(X)}<\infty.
$$

Nevertheless, the Hessians blow up. Let
$M_t=\sup_X\max_i|\lambda_i^\rho(\alpha_t)|$.

Let $\lambda_1\leq\cdots\leq\lambda_4$ be the ambient eigenvalues
and $\mu_1\leq\mu_2\leq\mu_3$ those on $TE$.
Interlacing, $\lambda_j\leq\mu_j\leq\lambda_{j+1}$, and
$\sum_j\arctan\lambda_j=\pi$ give
$|s-\pi|\leq\arctan M_t$ for $s=\sum_{j=1}^3\arctan\mu_j$.
Consequently,
\[
\sigma_2(\mu)-1
=-\cos s\prod_{j=1}^3\sqrt{1+\mu_j^2}\geq \cos (\arctan M_t)
\geq(1+M_t^2)^{-1/2},
\]
or, in terms of forms,
\[
\left.(3\alpha_t^2\wedge\rho-\rho^3)\right|_E
\geq(1+M_t^2)^{-1/2}\rho^3|_E.
\]

Integration gives
$$
3b_t^2-1\geq(1+M_t^2)^{-1/2},
$$
so $M_t\to\infty$ and hence
$\|\ddc u_t\|_{L^\infty(X,\rho)}\to\infty$.
{At $t=0$, the same restricted-form inequality for any smooth
solution would give a positive integral on $E$, whereas
$3b_0^2-1=0$. Thus no smooth endpoint solution exists.} The global $C^0$ estimate
persists at a boundary of smooth solvability within the K\"ahler
cone, without any symmetry assumption.
\end{example}

\begin{example}
\label{ex:inverse-sigma2-degeneration}
We discuss an inverse $\sigma_2$ equation example. Let $X=\operatorname{Bl}_p\mathbb P^4$, with $H$, $E$, and
$\eta=\pi^*\omega_{\mathrm{FS}}$ as above. Set $a=\sqrt{17/8}$
and fix a Calabi-symmetric K\"ahler form $\rho\in aH-E$.
For $0<t\ll1$, put
$$
\omega_t=\rho+(2-t-a)\eta\in(2-t)H-E,
\qquad
c_t=6\frac{a^2(2-t)^2-1}{(2-t)^4-1}
=3+3t+O(t^2).
$$
By \cite[Theorem~1.2]{MR3105762}, there are Calabi-symmetric
K\"ahler solutions
\begin{equation}\label{inversesigma2}
    \widehat\omega_t=\omega_t+\ddc u_t>0,\qquad
6\rho^2\wedge\widehat\omega_t^2=c_t\widehat\omega_t^4,
\qquad \sup_Xu_t=0.
\end{equation}

Set
$$
b_t=\sqrt{3/c_t},\qquad
\chi_t=\omega_t-b_t\rho
=(1-b_t)\rho+(2-t-a)\eta.
$$
These forms are K\"ahler and converge smoothly to the semipositive
big form $(2-a)\eta$. Note that the QSH condition can be checked directly.
At a lower contact point, put $A=b_t\rho+P$ with
$P=\chi_t+\ddc q>0$. Then $0<A\leq\widehat\omega_t$, so
$\sigma_2(\lambda_\rho(A)^{-1})\geq c_t$, or
$G_t(A):=c_t\det_\rho A-\sigma_2^\rho(A)\leq0$.
The first derivatives of $G_t$ vanish at $b_t\rho$ because
$c_tb_t^2=3$. With $Q=P/b_t$, direct expansion yields

$$
2\sigma_2^\rho(Q)+3\sigma_3^\rho(Q)+3\det_\rho Q\leq3.
$$
Maclaurin's inequalities consequently give
$$
\det_\rho P\leq b_t^4(\sqrt2-1)^4.
$$ We apply Corollary~\ref{cor:semipositive-uniform-Linfty}
to get
$$
\sup_{0<t\ll1}\|u_t\|_{L^\infty(X)}<\infty.
$$

Along $E$, the Calabi ansatz gives eigenvalues
$\nu_t,1,1,1$ for $\rho$ relative to $\widehat\omega_t$, where
$\nu_t=(c_t-3)/3=t+O(t^2)$. Hence
$$
P_t:=\chi_t+\ddc u_t=\widehat\omega_t-b_t\rho
$$
has an eigenvalue
\[
\nu_t^{-1}-b_t=t^{-1}+O(1)
\]
relative to $\rho$.
Consequently,
$$
\|\ddc u_t\|_{L^\infty(X,\rho)}\longrightarrow\infty.
$$
\end{example}

\appendix
\section{ABP Estimates under \texorpdfstring{$\mathcal D$}{D}-Subsolutions}
\label{ABP}
The following result is a variant of the ABP estimate established in \cite[Proposition 11]{MR3807322}. Although the statement there assumes the existence of a $\mathcal C$-subsolution, the argument only uses the pointwise determinant control \eqref{eq:D-subsolution-test} defining a $\mathcal D$-subsolution.

\begin{thm} \label{thm:ABP-explicit-delta} Let $(X,\rho)$ be a compact K\"ahler manifold and let $(F,\mathcal A)$ be an admissible pair satisfying QSH. Suppose $u$ is an admissible solution of \eqref{eq:main-equation}, normalized by $\sup_Xu=0$, and for some $\underline{u} \in C^2(X)$ and $\gamma\in C(X,\mathbb R)$ the triple $(\omega_{\underline u},\chi,\gamma)$ satisfies the $\mathcal D$-subsolution condition \eqref{eq:D-subsolution-test} of Definition~\ref{def:D-subsolution}, with $\chi\ge \delta\rho$ for some $\delta\in(0,1]$. Set
	\[
		B_{\underline u}:=b_{\mathcal A}
		+\|\operatorname{tr}_\rho\omega_{\underline u}\|_{L^\infty(X)}.
	\]
	Then for every $q>2$,
	\[
		\operatorname{osc}_X(u-\underline u)
		\le C(B_{\underline u},X,\rho,q)
		\left(\frac{\|e^\gamma\|_{L^q(X,\rho^n)}}{\delta^{\,n}}\right)^{\frac{2q}{q-2}}
		+C(X,\rho)\delta.
	\]
	If in addition $\sup_X\underline u=0$, then
	$\inf_X(u-\underline u)\le0\le\sup_X(u-\underline u)$, hence
	$\|u-\underline u\|_{L^\infty(X)}\le\operatorname{osc}_X(u-\underline u)$
	and the same right-hand side bounds $\|u-\underline u\|_{L^\infty(X)}$.
\end{thm}

\begin{proof}
	We follow \cite[Proposition~11]{MR3807322}; see also
	\autocite{MR2156505}. Set
	\[
		\widehat\omega:=\omega+\ddc\underline u=\omega_{\underline u},
		\qquad
		U:=u-\underline u-\sup_X(u-\underline u).
	\]
	Then $\sup_XU=0$, $\widehat\omega+\ddc U=\omega+\ddc u=\omega_u$,
	and $(\widehat\omega,\chi,\gamma)$ satisfies the
	$\mathcal D$-subsolution condition \eqref{eq:D-subsolution-test}
	with zero reference potential. We apply the argument below to $U$
	with background $\widehat\omega$; note that
	$\operatorname{osc}_XU=\operatorname{osc}_X(u-\underline u)$.
	Let $m:=\inf_XU$, attained at $x_0$. Choose holomorphic
	coordinates $(z^i)$ centered at $x_0$ on the unit ball $B_1$, and
	write $\beta:=dd^c|z|^2$. By compactness,
	$c_0\beta\le\rho\le C_0\beta$ on $B_1$ for uniform constants
	$c_0,C_0>0$. Since $\chi\ge\delta\rho\ge\delta c_0\beta$, set
	\[
		\varepsilon:=\frac{c_0}{2}\delta,
		\qquad v:=U+\varepsilon|z|^2,
	\]
	so that $\varepsilon\beta\le\chi$ on $B_1$.

	Let $\Gamma$ be the set of contact points where $v-p\cdot z$
	attains its minimum in $B_1$ for some $p\in\mathbb R^{2n}$ with
	$|p|\le\varepsilon/2$. By \cite[Proposition~10]{MR3807322},
	\[
		\D^2v\ge0,\qquad v\le m+\frac{\varepsilon}{2},\qquad
		dd^cv\ge0\quad\text{on }\Gamma.
	\]
	Set $P:=\chi+dd^cU=\chi-\varepsilon\beta+dd^cv$. Then
	$0\le dd^cv\le P$ on $\Gamma$, and since
	\[(\widehat\omega-\chi)+P
	=\widehat\omega+dd^cU=\omega_u\in\mathcal A_x\] with
	$F_x(\omega_u)=\psi(x)$, the $\mathcal D$-subsolution condition
	for $(\widehat\omega,\chi,\gamma)$ gives
	$\det_\rho(dd^cv)\le\det_\rho P\le e^{\gamma(x)}$ on
	$\Gamma$. Hence
	\[
		\det_{\mathbb R}\D^2v
		\le 2^{2n}\det_{\mathbb C}(v_{i\bar j})^2
		\le C(X,\rho)\,e^{2\gamma(x)}
		\quad\text{on }\Gamma.
	\]
	The ABP maximum principle and H\"older's inequality then give
	\[
		C(n)\varepsilon^{2n}
		\le C(X,\rho)\int_\Gamma e^{2\gamma}\,\rho^n
		\le C(X,\rho)\,\|e^\gamma\|_{L^q(X,\rho^n)}^2
			|\Gamma|_\rho^{\,1-\frac{2}{q}},
	\]
	so
	\[
		|\Gamma|_\rho
		\ge
		C(X,\rho,q)\,
		\frac{\delta^{2nq/(q-2)}}
		{\|e^\gamma\|_{L^q}^{2q/(q-2)}}.
	\]
	On the other hand, Lemma~\ref{lem:admissible-L1}, applied to
	$U$ with background $\widehat\omega$, gives
	\[
		\int_X(-U)\,\rho^n
		\le C(X,\rho)\bigl(b_{\mathcal A}
			+\|\operatorname{tr}_\rho\widehat\omega\|_{L^\infty}\bigr)
		= C(X,\rho)\,B_{\underline u},
	\]
	while $U\le m+\frac{\varepsilon}{2}=m+\frac{c_0\delta}{4}$ on
	$\Gamma$, hence $-U\ge-m-\frac{c_0\delta}{4}$ there, and
	\[
		C(X,\rho)\,B_{\underline u}
		\ge |\Gamma|_\rho\Bigl(-m-\frac{c_0\delta}{4}\Bigr).
	\]
	Combining the two estimates for $|\Gamma|_\rho$,
	\[
		-m
		\le
		C(B_{\underline u},X,\rho,q)
		\frac{\|e^\gamma\|_{L^q}^{2q/(q-2)}}
		{\delta^{2nq/(q-2)}}
		+C(X,\rho)\delta.
	\]
	Since $\sup_XU=0$, this bounds
	\[\|U\|_{L^\infty}=\operatorname{osc}_XU
	=\operatorname{osc}_X(u-\underline u).\]
	This completes the proof.
\end{proof}

\raggedbottom
\printbibliography

@article{GuoPhongEntropyEnergy,
  author  = {Guo, Bin and Phong, Duong H.},
  title   = {Uniform entropy and energy bounds for fully non-linear equations},
  journal = {Communications in Analysis and Geometry},
  volume  = {32},
  number  = {8},
  year    = {2024},
  pages   = {2305--2325},
  doi     = {10.4310/CAG.241212013839}
}

@article{Fang-MaGard,
      title={G{\aa}rding Polynomials}, 
      author={Hao Fang and Biao Ma},
      year={2026},
      eprint={2604.27755},
      archivePrefix={arXiv},
      primaryClass={math.CO},
      url={https://arxiv.org/abs/2604.27755}, 
}

@article{GurvitsLeake21,
  title = {Counting Matchings via Capacity-Preserving Operators},
  author = {Gurvits, Leonid and Leake, Jonathan},
  year = 2021,
  month = nov,
  journal = {Combinatorics, Probability and Computing},
  volume = {30},
  number = {6},
  pages = {956--981},
  issn = {0963-5483, 1469-2163},
  doi = {10.1017/S0963548321000122}
}

@article{AnariGharan21,
  title = {A Generalization of Permanent Inequalities and Applications in Counting and Optimization},
  author = {Anari, Nima and Oveis Gharan, Shayan},
  year = 2021,
  month = jun,
  journal = {Advances in Mathematics},
  volume = {383},
  pages = {107657},
  issn = {0001-8708},
  doi = {10.1016/j.aim.2021.107657},
}

@article{HarveyLawsonMajorization,
  author  = {Harvey, F. Reese and Lawson, Jr., H. Blaine},
  title   = {A definitive determinant majorization result
             for nonlinear operators},
  journal = {Duke Mathematical Journal},
  volume  = {174},
  number  = {13},
  pages   = {2749--2763},
  year    = {2025},
  doi     = {10.1215/00127094-2025-0003}
}

@article{Pham26,
  author        = {Doanh Pham},
  title         = {Determinant majorization for hyperbolic polynomials
                   on Euclidean Jordan algebras},
  journal       = {arXiv preprint},
  year          = {2026},
  eprint        = {2606.16713},
  archivePrefix = {arXiv}
}

@article {MR4913058,
	AUTHOR = {Qiao, Yuxiang},
	TITLE = {Sharp {$\rm{L}^\infty$} estimates for fully non-linear
	elliptic equations on compact complex manifolds},
	JOURNAL = {Calc. Var. Partial Differential Equations},
	FJOURNAL = {Calculus of Variations and Partial Differential Equations},
	VOLUME = {64},
	YEAR = {2025},
	NUMBER = {5},
	PAGES = {Paper No. 173, 39},
	ISSN = {0944-2669,1432-0835},
	MRCLASS = {32W50 (32Q15 32Q99 35J60 35J70)},
	MRNUMBER = {4913058},
	MRREVIEWER = {L\"uping\ Chen},
	DOI = {10.1007/s00526-025-03027-0},
	URL = {https://doi.org/10.1007/s00526-025-03027-0},
}

@incollection {MR2884031,
	AUTHOR = {Berman, Robert and Demailly, Jean-Pierre},
	TITLE = {Regularity of plurisubharmonic upper envelopes in big
	cohomology classes},
	BOOKTITLE = {Perspectives in analysis, geometry, and topology},
	SERIES = {Progr. Math.},
	VOLUME = {296},
	PAGES = {39--66},
	PUBLISHER = {Birkh\"auser/Springer, New York},
	YEAR = {2012},
	ISBN = {978-0-8176-8276-7},
	MRCLASS = {32U05 (32Q15 32U40 32W20)},
	MRNUMBER = {2884031},
	MRREVIEWER = {Vincent\ Guedj},
	DOI = {10.1007/978-0-8176-8277-4\_3},
	URL = {https://doi.org/10.1007/978-0-8176-8277-4_3},
}

@article {MR3191972,
	AUTHOR = {Demailly, Jean-Pierre and Dinew, S\l awomir and Guedj, Vincent
	and Pham, Hoang Hiep and Ko\l odziej, S\l awomir and Zeriahi,
	Ahmed},
	TITLE = {H\"older continuous solutions to {M}onge-{A}mp\`ere equations},
	JOURNAL = {J. Eur. Math. Soc. (JEMS)},
	FJOURNAL = {Journal of the European Mathematical Society (JEMS)},
	VOLUME = {16},
	YEAR = {2014},
	NUMBER = {4},
	PAGES = {619--647},
	ISSN = {1435-9855,1435-9863},
	MRCLASS = {32W20 (32Q15 32U05 32U15 32U40 35B65 35J96 53C55)},
	MRNUMBER = {3191972},
	MRREVIEWER = {Muhammed\ Ali\ Alan},
	DOI = {10.4171/JEMS/442},
	URL = {https://doi.org/10.4171/JEMS/442},
}

@article {MR2050205,
	AUTHOR = {Boucksom, S\'ebastien},
	TITLE = {Divisorial {Z}ariski decompositions on compact complex
	manifolds},
	JOURNAL = {Ann. Sci. \'Ecole Norm. Sup. (4)},
	FJOURNAL = {Annales Scientifiques de l'\'Ecole Normale Sup\'erieure.
	Quatri\`eme S\'erie},
	VOLUME = {37},
	YEAR = {2004},
	NUMBER = {1},
	PAGES = {45--76},
	ISSN = {0012-9593},
	MRCLASS = {32J18 (32C30)},
	MRNUMBER = {2050205},
	MRREVIEWER = {Adam\ Gregory\ Harris},
	DOI = {10.1016/j.ansens.2003.04.002},
	URL = {https://doi.org/10.1016/j.ansens.2003.04.002},
}

@article {MR2156505,
	AUTHOR = {B\l ocki, Zbigniew},
	TITLE = {On uniform estimate in {C}alabi-{Y}au theorem},
	JOURNAL = {Sci. China Ser. A},
	FJOURNAL = {Science in China. Series A. Mathematics},
	VOLUME = {48},
	YEAR = {2005},
	PAGES = {244--247},
	ISSN = {1006-9283,1862-2763},
	MRCLASS = {32W20 (32Q25)},
	MRNUMBER = {2156505},
	MRREVIEWER = {S\l awomir\ Ko\l odziej},
	DOI = {10.1007/BF02884710},
	URL = {https://doi.org/10.1007/BF02884710},
}

@article {MR480350,
	AUTHOR = {Yau, Shing Tung},
	TITLE = {On the {R}icci curvature of a compact {K}\"ahler manifold and
	the complex {M}onge-{A}mp\`ere equation. {I}},
	JOURNAL = {Comm. Pure Appl. Math.},
	FJOURNAL = {Communications on Pure and Applied Mathematics},
	VOLUME = {31},
	YEAR = {1978},
	NUMBER = {3},
	PAGES = {339--411},
	ISSN = {0010-3640,1097-0312},
	MRCLASS = {53C55 (32C10 35J60)},
	MRNUMBER = {480350},
	MRREVIEWER = {Robert\ E.\ Greene},
	DOI = {10.1002/cpa.3160310304},
	URL = {https://doi.org/10.1002/cpa.3160310304},
}

@article {MR4567566,
	AUTHOR = {Sui, Zhenan and Sun, Wei},
	TITLE = {On {$L^\infty$} estimate for complex {H}essian quotient
	equations on compact {K}\"ahler manifolds},
	JOURNAL = {J. Geom. Anal.},
	FJOURNAL = {Journal of Geometric Analysis},
	VOLUME = {33},
	YEAR = {2023},
	NUMBER = {6},
	PAGES = {Paper No. 165, 26},
	ISSN = {1050-6926,1559-002X},
	MRCLASS = {58J05 (32W20)},
	MRNUMBER = {4567566},
	MRREVIEWER = {Freid\ Tong},
	DOI = {10.1007/s12220-023-01220-1},
	URL = {https://doi.org/10.1007/s12220-023-01220-1},
}

@misc{guo2024uniformlinftyestimatessubsolutions,
	title={Uniform $L^\infty$ estimates: subsolutions to fully nonlinear partial differential equations}, 
	author={Bin Guo and Duong H. Phong},
	year={2024},
	eprint={2401.11572},
	archivePrefix={arXiv},
	primaryClass={math.AP},
	url={https://arxiv.org/abs/2401.11572}, 
}

@book{Murota03,
author = {Murota, Kazuo},
title = {Discrete Convex Analysis},
publisher = {Society for Industrial and Applied Mathematics},
year = {2003},
doi = {10.1137/1.9780898718508},
address = {},
edition   = {},
URL = {https://epubs.siam.org/doi/abs/10.1137/1.9780898718508},
%eprint = {https://epubs.siam.org/doi/pdf/10.1137/1.9780898718508}
}

@article{MR2200852,
	author = {Gurvits, Leonid},
	doi = {10.1016/j.aim.2004.12.002},
	fjournal = {Advances in Mathematics},
	issn = {0001-8708,1090-2082},
	journal = {Adv. Math.},
	mrclass = {15A15},
	mrnumber = {2200852},
	mrreviewer = {My\ Hachem\ Lalaoui Rhali},
	number = {2},
	pages = {435--454},
	title = {The van der {W}aerden conjecture for mixed discriminants},
	url = {https://doi.org/10.1016/j.aim.2004.12.002},
	volume = {200},
	year = {2006}}

@article{BrandenHuh20,
	author = {Br{\"a}nd{\'e}n, Petter and Huh, June},
	doi = {10.4007/annals.2020.192.3.4},
	fjournal = {Annals of Mathematics. Second Series},
	issn = {0003-486X},
	journal = {Ann. Math. (2)},
	number = {3},
	pages = {821--891},
	title = {Lorentzian polynomials},
	volume = {192},
	year = {2020},
	zbl = {1454.52013},
	zbmath = {7285355}}

@article {MR2746347,
    AUTHOR = {Boucksom, S\'ebastien and Eyssidieux, Philippe and Guedj,
              Vincent and Zeriahi, Ahmed},
     TITLE = {Monge-{A}mp\`ere equations in big cohomology classes},
   JOURNAL = {Acta Math.},
  FJOURNAL = {Acta Mathematica},
    VOLUME = {205},
      YEAR = {2010},
    NUMBER = {2},
     PAGES = {199--262},
      ISSN = {0001-5962,1871-2509},
   MRCLASS = {32U40 (32Q20 32U15 32W20)},
  MRNUMBER = {2746347},
MRREVIEWER = {S\l awomir\ Dinew},
       DOI = {10.1007/s11511-010-0054-7},
       URL = {https://doi.org/10.1007/s11511-010-0054-7},
}

@incollection {MR270945,
    AUTHOR = {Edmonds, Jack},
     TITLE = {Submodular functions, matroids, and certain polyhedra},
 BOOKTITLE = {Combinatorial {S}tructures and their {A}pplications ({P}roc.
              {C}algary {I}nternat. {C}onf., {C}algary, {A}lta., 1969)},
     PAGES = {69--87},
 PUBLISHER = {Gordon and Breach, New York-London-Paris},
      YEAR = {1970},
   MRCLASS = {05.35},
  MRNUMBER = {270945},
MRREVIEWER = {A.\ W.\ Ingleton},
}

@article {MR4781476,
    AUTHOR = {Fang, Hao and Ma, Biao},
     TITLE = {On a fully nonlinear elliptic equation with differential
              forms},
   JOURNAL = {Adv. Math.},
  FJOURNAL = {Advances in Mathematics},
    VOLUME = {454},
      YEAR = {2024},
     PAGES = {Paper No. 109867, 89},
      ISSN = {0001-8708,1090-2082},
   MRCLASS = {53C55},
  MRNUMBER = {4781476},
MRREVIEWER = {Martin\ de Borbon},
       DOI = {10.1016/j.aim.2024.109867},
       URL = {https://doi.org/10.1016/j.aim.2024.109867},
}

@book {MR2171629,
    AUTHOR = {Fujishige, Satoru},
     TITLE = {Submodular functions and optimization},
    SERIES = {Annals of Discrete Mathematics},
    VOLUME = {58},
   EDITION = {Second},
 PUBLISHER = {Elsevier B. V., Amsterdam},
      YEAR = {2005},
     PAGES = {xiv+395},
      ISBN = {0-444-52086-4},
   MRCLASS = {90C10 (05A18 05B35 52A41 90C27 90C35 90C57)},
  MRNUMBER = {2171629},
}

@article {MR3284698,
    AUTHOR = {Guan, Bo},
     TITLE = {Second-order estimates and regularity for fully nonlinear
              elliptic equations on {R}iemannian manifolds},
   JOURNAL = {Duke Math. J.},
  FJOURNAL = {Duke Mathematical Journal},
    VOLUME = {163},
      YEAR = {2014},
    NUMBER = {8},
     PAGES = {1491--1524},
      ISSN = {0012-7094,1547-7398},
   MRCLASS = {35R01 (35B65 35J60 35J96 58J05)},
  MRNUMBER = {3284698},
MRREVIEWER = {Pierpaolo\ Esposito},
       DOI = {10.1215/00127094-2713591},
       URL = {https://doi.org/10.1215/00127094-2713591},
}

@article {MR4727578,
    AUTHOR = {Guan, Bo},
     TITLE = {On subsolutions and concavity for fully nonlinear elliptic
              equations},
   JOURNAL = {Adv. Nonlinear Stud.},
  FJOURNAL = {Advanced Nonlinear Studies},
    VOLUME = {24},
      YEAR = {2024},
    NUMBER = {1},
     PAGES = {15--28},
      ISSN = {1536-1365,2169-0375},
   MRCLASS = {35J15 (35B45 35J60 35R01 58J05)},
  MRNUMBER = {4727578},
       DOI = {10.1515/ans-2023-0116},
       URL = {https://doi.org/10.1515/ans-2023-0116},
}

@article {MR4593734,
    AUTHOR = {Guo, Bin and Phong, Duong H. and Tong, Freid},
     TITLE = {On {$L^\infty$} estimates for complex {M}onge-{A}mp\`ere
              equations},
   JOURNAL = {Ann. of Math. (2)},
  FJOURNAL = {Annals of Mathematics. Second Series},
    VOLUME = {198},
      YEAR = {2023},
    NUMBER = {1},
     PAGES = {393--418},
      ISSN = {0003-486X,1939-8980},
   MRCLASS = {35J60 (32W20 35J96 53C55 53C56)},
  MRNUMBER = {4593734},
       DOI = {10.4007/annals.2023.198.1.4},
       URL = {https://doi.org/10.4007/annals.2023.198.1.4},
}

@article {MR2411443,
    AUTHOR = {Gurvits, Leonid},
     TITLE = {Van der {W}aerden/{S}chrijver-{V}aliant like conjectures and
              stable (aka hyperbolic) homogeneous polynomials: one theorem
              for all},
      NOTE = {With a corrigendum},
   JOURNAL = {Electron. J. Combin.},
  FJOURNAL = {Electronic Journal of Combinatorics},
    VOLUME = {15},
      YEAR = {2008},
    NUMBER = {1},
     PAGES = {Research Paper 66, 26},
      ISSN = {1077-8926},
   MRCLASS = {15A15 (05A15 05C70)},
  MRNUMBER = {2411443},
MRREVIEWER = {Peter\ M.\ Gibson},
       DOI = {10.37236/790},
       URL = {https://doi.org/10.37236/790},
}

@article {MR4589710,
    AUTHOR = {Harvey, F. Reese and Lawson, Jr., H. Blaine},
     TITLE = {Determinant majorization and the work of {G}uo-{P}hong-{T}ong
              and {A}bja-{O}live},
   JOURNAL = {Calc. Var. Partial Differential Equations},
  FJOURNAL = {Calculus of Variations and Partial Differential Equations},
    VOLUME = {62},
      YEAR = {2023},
    NUMBER = {5},
     PAGES = {Paper No. 153, 28},
      ISSN = {0944-2669,1432-0835},
   MRCLASS = {58J52 (53A55)},
  MRNUMBER = {4589710},
MRREVIEWER = {S\l awomir\ Dinew},
       DOI = {10.1007/s00526-023-02485-8},
       URL = {https://doi.org/10.1007/s00526-023-02485-8},
}

@article {MR4965192,
    AUTHOR = {Reese Harvey, F. and Blaine Lawson, Jr., H.},
     TITLE = {A definitive determinant majorization result for nonlinear
              operators},
   JOURNAL = {Duke Math. J.},
  FJOURNAL = {Duke Mathematical Journal},
    VOLUME = {174},
      YEAR = {2025},
    NUMBER = {13},
     PAGES = {2749--2763},
      ISSN = {0012-7094,1547-7398},
   MRCLASS = {32W20 (35A23 35G20 53C55 58J32)},
  MRNUMBER = {4965192},
MRREVIEWER = {John\ Urbas},
       DOI = {10.1215/00127094-2025-0003},
       URL = {https://doi.org/10.1215/00127094-2025-0003},
}

@article {MR1618325,
    AUTHOR = {Ko\l odziej, S\l awomir},
     TITLE = {The complex {M}onge-{A}mp\`ere equation},
   JOURNAL = {Acta Math.},
  FJOURNAL = {Acta Mathematica},
    VOLUME = {180},
      YEAR = {1998},
    NUMBER = {1},
     PAGES = {69--117},
      ISSN = {0001-5962,1871-2509},
   MRCLASS = {32F07 (32C17 35J60)},
  MRNUMBER = {1618325},
MRREVIEWER = {M.\ Klimek},
       DOI = {10.1007/BF02392879},
       URL = {https://doi.org/10.1007/BF02392879},
}

@article {MR3807322,
    AUTHOR = {Sz\'ekelyhidi, G\'abor},
     TITLE = {Fully non-linear elliptic equations on compact {H}ermitian
              manifolds},
   JOURNAL = {J. Differential Geom.},
  FJOURNAL = {Journal of Differential Geometry},
    VOLUME = {109},
      YEAR = {2018},
    NUMBER = {2},
     PAGES = {337--378},
      ISSN = {0022-040X,1945-743X},
   MRCLASS = {58J05 (32W50 35J60 35J96 35R01 53C55)},
  MRNUMBER = {3807322},
MRREVIEWER = {Bianca\ Santoro},
       DOI = {10.4310/jdg/1527040875},
       URL = {https://doi.org/10.4310/jdg/1527040875},
}

@book {MR1956924,
    AUTHOR = {Schrijver, Alexander},
     TITLE = {Combinatorial optimization. {P}olyhedra and efficiency. {V}ol.
              {A}},
    SERIES = {Algorithms and Combinatorics},
    VOLUME = {24,A},
      NOTE = {Paths, flows, matchings,
              Chapters 1--38},
 PUBLISHER = {Springer-Verlag, Berlin},
      YEAR = {2003},
     PAGES = {xxxviii+647},
      ISBN = {3-540-44389-4},
   MRCLASS = {90-02 (05-02 52B55 68Q25 68R10 90C27 90C35 90C57)},
  MRNUMBER = {1956924},
MRREVIEWER = {Alexander\ I.\ Barvinok},
}

@book {MR1956925,
    AUTHOR = {Schrijver, Alexander},
     TITLE = {Combinatorial optimization. {P}olyhedra and efficiency. {V}ol.
              {B}},
    SERIES = {Algorithms and Combinatorics},
    VOLUME = {24,B},
      NOTE = {Matroids, trees, stable sets,
              Chapters 39--69},
 PUBLISHER = {Springer-Verlag, Berlin},
      YEAR = {2003},
     PAGES = {i--xxxiv and 649--1217},
      ISBN = {3-540-44389-4},
   MRCLASS = {90-02 (05-02 52B55 68Q25 68R10 90C27 90C35 90C57)},
  MRNUMBER = {1956925},
MRREVIEWER = {Alexander\ I.\ Barvinok},
}

@book {MR1956926,
    AUTHOR = {Schrijver, Alexander},
     TITLE = {Combinatorial optimization. {P}olyhedra and efficiency. {V}ol.
              {C}},
    SERIES = {Algorithms and Combinatorics},
    VOLUME = {24,C},
      NOTE = {Disjoint paths, hypergraphs,
              Chapters 70--83},
 PUBLISHER = {Springer-Verlag, Berlin},
      YEAR = {2003},
     PAGES = {i--xxxiv and 1219--1881},
      ISBN = {3-540-44389-4},
   MRCLASS = {90-02 (05-02 52B55 68Q25 68R10 90C27 90C35 90C57)},
  MRNUMBER = {1956926},
MRREVIEWER = {Alexander\ I.\ Barvinok},
}

@book {MR3617346,
    AUTHOR = {Guedj, Vincent and Zeriahi, Ahmed},
     TITLE = {Degenerate complex {M}onge-{A}mp\`ere equations},
    SERIES = {EMS Tracts in Mathematics},
    VOLUME = {26},
 PUBLISHER = {European Mathematical Society (EMS), Z\"urich},
      YEAR = {2017},
     PAGES = {xxiv+472},
      ISBN = {978-3-03719-167-5},
   MRCLASS = {32W20 (32Q20 32U15 32U20 32U40 35J96)},
  MRNUMBER = {3617346},
MRREVIEWER = {Slimane\ Benelkourchi},
       DOI = {10.4171/167},
       URL = {https://doi.org/10.4171/167},
}

@article {MR3105762,
    AUTHOR = {Fang, Hao and Lai, Mijia},
     TITLE = {Convergence of general inverse {$\sigma_k$}-flow on {K}\"ahler
              manifolds with {C}alabi ansatz},
   JOURNAL = {Trans. Amer. Math. Soc.},
  FJOURNAL = {Transactions of the American Mathematical Society},
    VOLUME = {365},
      YEAR = {2013},
    NUMBER = {12},
     PAGES = {6543--6567},
      ISSN = {0002-9947,1088-6850},
   MRCLASS = {53C44 (32W20 53C55)},
  MRNUMBER = {3105762},
MRREVIEWER = {Jeffrey\ D.\ Streets},
       DOI = {10.1090/S0002-9947-2013-05947-8},
       URL = {https://doi.org/10.1090/S0002-9947-2013-05947-8},
}

@article{fu2026criticallyzequationkahler,
      title={The Critical LYZ Equation in K\"ahler Geometry}, 
      author={Jixiang Fu and Shing-Tung Yau and Dekai Zhang},
      year={2026},
      eprint={2511.21492},
      archivePrefix={arXiv},
      primaryClass={math.DG},
      url={https://arxiv.org/abs/2511.21492}, 
}

@misc{fang2026liouvillerigidityrealcomplex,
      title={Liouville Rigidity for Real and Complex Degenerate Hessian Equations}, 
      author={Hao Fang and Biao Ma and Jinyang Wu},
      year={2026},
      eprint={2607.21024},
      archivePrefix={arXiv},
      primaryClass={math.AP},
      url={https://arxiv.org/abs/2607.21024}, 
}

@misc{fang2026idealgaardingpolynomials,
      title={Ideal G{\aa}rding polynomials}, 
      author={Hao Fang and Biao Ma},
      year={2026},
      eprint={2607.16832},
      archivePrefix={arXiv},
      primaryClass={math.CO},
      url={https://arxiv.org/abs/2607.16832}, 
}
\end{document}